\documentclass[microtype]{gtpart}

\usepackage[english]{babel}
\usepackage{pstricks-add}
\usepackage{array}
\usepackage{amsfonts}
\usepackage{graphics}
\usepackage[T1]{fontenc}
\usepackage[utf8]{inputenc}
\usepackage{rotating}
\usepackage[all]{xy}
\usepackage{amsmath}
\usepackage{amssymb}
\usepackage{amsthm}
\usepackage{nicefrac}
\usepackage{amscd}
\usepackage{graphicx}

\title[Ellis enveloping semigroup for almost canonical model sets]{Ellis enveloping semigroup for almost \\ canonical model sets of an Euclidean space}

\author{Jean-Baptiste Aujogue}
\givenname{Jean-baptiste}
\surname{Aujogue}
\address{Universidad de Santiago de Chile\\Dep. de Matem\'aticas, Fac. de Ciencia\\
Aladema 3363, Estaci\'on Central\\
Santiago\\
Chile}
\email{jean.baptiste@usach.cl}


\arxivreference{}
\arxivpassword{}

\newtheorem{theo}{Theorem}[section]
\newtheorem{de}[theo]{Definition}
\newtheorem{prop}[theo]{Proposition}
\newtheorem{lem}[theo]{Lemma}

\newtheorem{as}{Assumption}

\volumenumber{}
\issuenumber{}
\publicationyear{}
\papernumber{}
\startpage{}
\endpage{}
\doi{}
\MR{}
\Zbl{}
\received{}
\revised{}
\accepted{}
\published{}
\publishedonline{}
\proposed{}
\seconded{}
\corresponding{}
\editor{}
\version{}

\begin{document}

\begin{abstract} We consider certain point patterns of an Euclidean space and calculate the Ellis enveloping semigroup of their associated dynamical systems. The algebraic structure and the topology of the Ellis semigroup, as well as its action on the underlying space, are explicitly described. As an example, we treat the vertex pattern of the Amman-Beenker tiling of the plane.
\end{abstract}

\maketitle

\section*{\large{\textbf{Introduction}}}


This article proposes to study certain aspects of dynamical systems associated with \textit{point patterns} of Euclidean space. The topic of point patterns arose in symbolic dynamics, and also concerns aperiodic tilings. Point patterns have been studied by numerous authors for the last thirty years after the discovery by Schetchmann et al of the physical materials now commonly called \textit{quasicrystals}. In this context, a point pattern of a Euclidean space $\mathbb{R}^d$ is thought of as an alloy, where points are understood as positions of atoms (or molecules or electrons) and the quasicrystalline structure then arises when a certain \textit{long range order} is observed on the disposition of points within the pattern.




\vspace{0.2cm}

A great success in the topic of point patterns is the possibility to handle a pattern $\Lambda _0$ of $\mathbb{R}^d$ by considering the \textit{dynamical system} associated to it. The system consists of a space $\mathbb{X}_{\Lambda_0}$, called the \textit{hull} of $\Lambda_0$, which is formed of all other point patterns that locally look like $\Lambda_0$ and is endowed with a suitable compact topology, together with an action of the space $\mathbb{R}^d$ by homeomorphisms. Natural properties of a pattern of geometric, combinatoric or statistical nature are then displayed by topological, dynamical or ergodic features in this dynamical system. This is particularly true for long range order on point patterns, where the counterpart seems to rely on the existence of \textit{eigenfunctions} for the associated dynamical system. For instance, within the class of substitutive point patterns, the \textit{Meyer property}, which is a strong form of internal order (Moody \cite{Mo}), is equivalent to the existence of a nontrivial eigenfunction for the associated dynamical system (Lee and Solomyak \cite{LeeSo}). This type of statement also exists outside the realm of substitution patterns (Kellendonk and Sadun \cite{KeSa}). Another example concerns the subclass of \textit{model sets}, which can be viewed as the most ordered aperiodic patterns. The property of being a model set is equivalent to being a Meyer set such that continuous eigenfunctions separate a residual subset of elements in its associated hull (see \cite{Auj}, Baake, Lenz and Moody \cite{BaaLenMo}, Lee and Moody \cite{LeeMo}). In other words, model sets are exactly the point patterns with the Meyer property and \textit{almost automorphic} associated dynamical system. A third striking result is that \textit{pure point diffractivity} of a pattern (Hof \cite{Hof}), with is truly of statistical nature (Moody \cite{Mo3}), is known to be equivalent to the existence of a basis of eigenfunctions for the Hilbert space provided by the hull together with a certain ergodic measure (there is a widely developed literature about this aspect of patterns, see for instance Lee, Moody and Solomyak \cite{LeeMoSo} and Baake and Lenz \cite{BaaLe} and references therein). These statements have been proven under various mild assumptions on the pattern considered.

\vspace{0.2cm}
A certain form of this eigenvalue problem for a point pattern can be addressed, from a topological point of view, by the knowledge of the \textit{Ellis enveloping semigroup} of its dynamical system $(\mathbb{X}, \mathbb{R}^d)$. This semigroup was introduced for dynamical systems by Ellis and Gottschalk \cite{EllGo} as a way to study actions of a group on a compact space from an algebraic point of view. In a series of papers, Glasner investigated this semigroup for fairly general dynamical systems (see the review Glasner \cite{Gl} and references therein), and he and Megrelishvili showed in \cite{GlMe} that the Ellis semigroup $E(\mathbb{X}, \mathbb{R}^d)$ obeys a dichotomy: It is either the sequential closure of the acting group $\mathbb{R}^d$ or contains a topological copy of the Stone-\v{C}ech compactification $\beta \mathbb{N}$ of the integers. The former situation admits several equivalent formulations, and when it occurs the underlying dynamical system is called \textit{tame}; see Glasner \cite{Gl0}. Tame systems are dynamically simple: Indeed it is proved in Glasner \cite{Gl2} that they are uniquely ergodic, almost automorphic and measurably conjugated with a Kronecker system. In particular, a point pattern with the Meyer property and a tame dynamical system must be a model set.

\vspace{0.2cm}

In this work we propose a qualitative description of the Ellis semigroup of dynamical systems associated with particular point patterns, the \textit{almost canonical model sets}. These particular patterns are relevant in the crystallographic sense, as well as very accessible mathematically: One can get a complete picture of the hull $\mathbb{X}_{\Lambda_0}$ of such patterns (Le \cite{Le}), as well as their associated $C^*$-Algebras (a recent source is Putnam \cite{Put}, see references therein), and also compute their cohomology and K-theory groups (Forrest, Hunton and Kellendonk \cite{FoHuKe}, G\"{a}hler, Hunton and Kellendonk \cite{GaHuKe} and Putnam \cite{Put}) as well as the asymptotic exponent of their complexity function (Julien \cite{Ju}). We show that in our situation it is possible to completely describe elements of the Ellis semigroup, their action onto the underlying space, as well as the algebraic and topological structure of this semigroup. The type of calculations made here can be compared with the calculations performed in Pikula \cite{Pik} about Sturmian and Sturmian-like systems (see also Glasner \cite[Example 4.5]{Gl}). We also show that for those dynamical systems the Ellis semigroup is of first class in the sense of the dichotomy of Glasner and Megrelishvili \cite{GlMe}, that is, almost canonical model sets have tame systems.

\section*{\large{\textbf{The contents of the paper}}}


To construct a model set of $\mathbb{R}^d$, one begins by considering a higher dimensional Euclidean space $\mathbb{R}^{n+d}$, together with a lattice $\Sigma$ in it, as well as an embedded $d$-dimensional slope, usually placed in an "irrational" manner, which is thought as the space $\mathbb{R}^d$ itself. Such an environment used to construct a model set is called a \textit{cut and project scheme}. The second step is to consider a suitable region $W$ of the Euclidean subspace $\mathbb{R}^n$ orthogonal to $\mathbb{R}^d$. The model set in question thereby emerges as the orthogonal projection to $\mathbb{R}^d$ of certain points in $\Sigma$, namely those that fall into the region $W$ when projected orthogonally onto $\mathbb{R}^n$. A model set is thus written as
\begin{align*} \Lambda _0:= \left\lbrace  \gamma ^{\Vert } \; \vert \; \gamma \in \Sigma \; \text{ and } \; \gamma ^{\perp} \in W \right\rbrace ,
\end{align*}
where $^{\Vert }$ and $^{\perp}$ denote the orthogonal projections onto $\mathbb{R}^d$ and $\mathbb{R}^n$ respectively. In the above context we will speak about a \textit{real model set} (see the discussion in Section $1$), the word "real" coming from the fact that the summand $\mathbb{R}^n$ used here to form the cut and project scheme is a Euclidean space.


\vspace{0.2cm}

The dynamical system $(\mathbb{X},\mathbb{R}^d)$ associated with a real model set is of very particular form: It is an \textit{almost automorphic} extension over a torus $$\mathbb{T}^{n+d}:=\mathbb{R}^{n+d}\diagup _{\Sigma }$$ (see the material of Section $2$). This property will be central in our task, and shows up in the consideration of a certain factor map, known as the \textit{parametrization map} (Baake, Hermisson and Pleasant \cite{BaaHePl} and Schlottmann \cite{Sch}),
\begin{align*}
\pi : \mathbb{X} \longrightarrow \mathbb{T}^{n+d}.
\end{align*}
This mapping also demonstrates that any pattern $\Lambda$ in the hull $\mathbb{X}$ of a model set $\Lambda _0$ is also a model set, such that if $[w,t]_\Sigma \in \mathbb{R}^{n+d}\diagup _{\Sigma }=\mathbb{T}^{n+d}$ is its image, then $\Lambda$ is determined, as model set, by the region $W+w$ in $\mathbb{R}^n$, next translated by the vector $t$ in $\mathbb{R}^d$. This is described in better details in first section of the main text.



\vspace{0.2cm}

The first step in determining the Ellis semigroup $E(\mathbb{X},\mathbb{R}^d)$ is to describe it as a \textit{suspension} of another (simpler) semigroup (see Section $3$). To that end we let $\Gamma$ be the subgroup of $\mathbb{R}^d$ obtained as the orthogonal projection of the lattice $\Sigma$ used to construct $\Lambda_0$ as a model set. $\Gamma$ is generally not a lattice in $\mathbb{R}^d$, and will even often be dense in $\mathbb{R}^d$, although it is always finitely generated. We now consider the collection $\Xi ^\Gamma$ of point patterns in $\mathbb{X}$ that are contained, as subsets of $\mathbb{R}^d$, in $\Gamma$. This subset of $\mathbb{X}$ remains stable under the action of any vector of $\mathbb{R}^d$ which lies in $\Gamma$, and when endowed with a suitable topology it gives rise to a new dynamical system $(\Xi ^\Gamma , \Gamma)$. We call this latter system the \textit{subsystem} associated with $\Lambda _0$. The space $\Xi ^\Gamma $ will have a locally compact totally disconnected topology, and as a result its Ellis enveloping semigroup $E(\Xi ^\Gamma , \Gamma)$ will be a locally compact totally disconnected topological space (for Ellis semigroup of dynamical systems over locally compact spaces, see Section $2$). The importance of this semigroup in our setting is highlighted by Theorem \ref{theo:suspension.ellis}, which yields a algebraic isomorphism and homeomorphism 
\begin{align*}
E(\mathbb{X},\mathbb{R}^d) \simeq  E(\Xi ^\Gamma , \Gamma)\times _{\Gamma}\mathbb{R}^d ,
\end{align*}
where the right-hand term is understood as a quotient of $E(\Xi ^\Gamma , \Gamma)\times \mathbb{R}^d$ under a natural diagonal action of $\Gamma$. This theorem shows in particular that the Ellis semigroup $E(\mathbb{X},\mathbb{R}^d)$ is in our context a \textit{matchbox manifold}: It is locally the product of an Euclidean open subset with a totally disconnected space. It also asserts that the non-trivial (and in particular the noncommutative) part of $E(\mathbb{X},\mathbb{R}^d)$ is displayed by the semigroup $E(\Xi ^\Gamma , \Gamma)$.

\vspace{0.2cm}

We will thus from now on focus on the calculation of $E(\Xi ^\Gamma , \Gamma)$. At first, we show the existence of an onto continuous semigroup morphism 
\begin{align*}
\Pi ^* : E(\Xi ^\Gamma , \Gamma) \longrightarrow \mathbb{R}^n .
\end{align*}


This morphism is closely related with the parametrization map presented above, and will allow us to understand the convergence of a net in $E(\Xi ^\Gamma , \Gamma)$ by studying how the corresponding net, via this morphism, converges in $\mathbb{R}^n$.

\vspace{0.2cm}

Our wish is to find a certain semigroup $S$, together with a certain semigroup morphism from $E(\Xi ^\Gamma , \Gamma)$ into $S$, such that the Ellis semigroup $E(\Xi ^\Gamma , \Gamma)$ embeds in the direct product $S\times \mathbb{R}^n$. To simplify the problem we let the \textit{almost canonical property} enter the game. This property consists of a condition on the region $W$ used to obtain $\Lambda _0$ as model set, that is, $W$ must be a polytope of $\mathbb{R}^n$ satisfying a particular condition (see Section $4$). Under the assumption that the region $W$ is almost canonical, together with the almost automorphic property observed on the dynamical system $(\mathbb{X}_{\Lambda _0}, \mathbb{R}^d)$, we are able to identify the correct semigroup $S$ as the \textit{face semigroup} associated with the polytope $W$ in $\mathbb{R}^n$ (see Sections $5$ and $6$ for our presentation and results).

\vspace{0.2cm}

We may briefly present the face semigroup $\mathfrak{T}_{W}$ associated with the polytope $W$ in $\mathbb{R}^n$ as follows: The polytope $W$ determines a finite collection of linear hyperplanes $\mathfrak{H}_{W}$ in $\mathbb{R}^n$, namely the ones that are parallel to at least one face of $W$. This collection in turn determines a stratification of $\mathbb{R}^n$ by cones, all being, for each hyperplane $H\in \mathfrak{H}_W$, included in $H$ or integrally part of one of the two possible complementary half-spaces. An illustration of this construction is provided in Section $7$, where $W$ is a regular octagon in $\mathbb{R}^2$. The face semigroup $\mathfrak{T}_{W}$ is set-theoretically the finite collection of cones resulting from this stratification process, together with a (noncommutative) semigroup product stating that the product $C.C'$ of two cones is the cone which the head of $C$ enters after being translated by small vectors of $C'$. The elements of $\mathfrak{T}_W$ are more conveniently described as "side maps", which consist of mappings from $\mathfrak{H}_W$ to the three symbol set $\lbrace -, 0,+\rbrace$, giving the relative position of any cone with respect to each hyperplane. This formalism has the advantage to allows for a concise and handy formulation of the product law on this semigroup (see Section $6$).

\vspace{0.2cm}
The embedding morphism
\begin{align*}
E(\Xi ^\Gamma , \Gamma) \longrightarrow  \mathfrak{H}_W\times \mathbb{R}^n
\end{align*}

comes from the observation that a neighborhood basis of any transformation $g\in E(\Xi ^\Gamma , \Gamma)$ is provided by the vector $w_g:= \Pi ^*(g)$ of $\mathbb{R}^n$, together with a certain cone $$C_g\in \mathfrak{T}_W ,$$ in the sense that a net in $\Gamma$ converges to $g$ in the Ellis semigroup $E(\Xi ^\Gamma , \Gamma)$ (such a net exists by construction) if and only if the corresponding net in $\mathbb{R}^n$ converges to $w_g$ and eventually lies in $C_g+w_g$. In this sense, the cone $C_g$ provides the direction a net must follow in order to converge to the transformation $g$. This allows us to calculate the corresponding image subsemigroup in $\mathfrak{H}_W\times \mathbb{R}^n$, which is the aim of Theorem \ref{theo:principal.interne}, and proves to be a finite disjoint union of \textit{subgroups} of $\mathbb{R}^n$. Moreover the topology of $E(\Xi ^\Gamma , \Gamma)$ is completely described by a geometric criterion of convergence for nets.

\vspace{0.2cm}

Finally, we fusion Theorems \ref{theo:suspension.ellis} and \ref{theo:principal.interne} to formulate our main theorem \ref{theo:principal}, which establishes the existence of an embedding semigroup morphism
\begin{align*}
E(\mathbb{X} , \mathbb{R}^d) \longrightarrow  \mathfrak{H}_W\times \mathbb{T}^{n+d}
\end{align*}


for which the image subsemigroup together with its topology are identified. Interestingly, this semigroup remains exactly the same for model sets issued after translating, dilating, or deforming the region $W$ as long as the hyperplanes determined by the faces are unchanged. As a byproduct of the previous analysis we show that the topology of the Ellis semigroup $E(\mathbb{X},\mathbb{R}^d)$ admits a first countable topology, and thus is the sequential closure of the acting group $\mathbb{R}^d$. We conclude this work by determining some algebraic features of this Ellis semigroup, as well as a picture of its underlying action on the space $\mathbb{X}$.

\section{\large{\textbf{Model sets and associated dynamical systems}}}\label{model sets}

\subsection{General definition of inter-model set}

To define what an (almost canonical) model set in $\mathbb{R}^d$ is (see \cite{Sch}, as well as \cite{Mo2} for a more detailed exposition), we consider first an environment used to construct it, namely a \textit{cut and project scheme}. This consists of a triple $(\mathsf{H}, \Sigma , \mathbb{R}^d)$ and a diagram

\includegraphics[trim = 4cm 22.8cm 5.5cm 5.2cm]{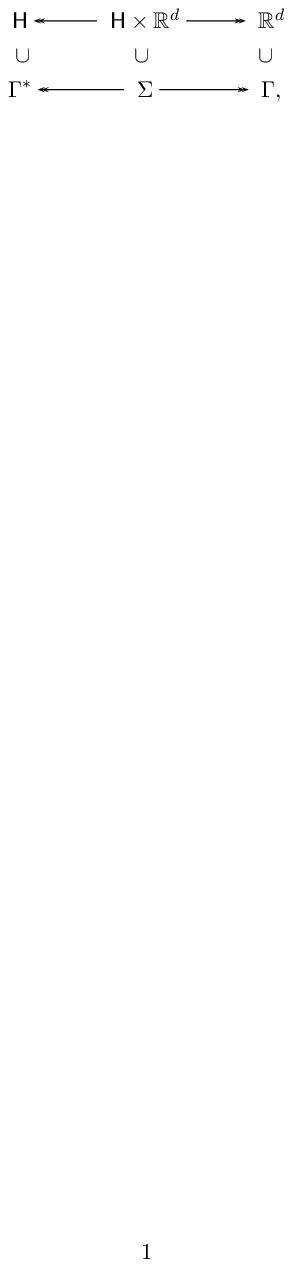}

where $\mathsf{H}$ is a locally compact Abelian group and

\begin{itemize}
\item[•] $\Sigma$ is a countable lattice in $\mathsf{H}\small{\times}\mathbb{R}^{d}$, ie a countable discrete and cocompact subgroup,
\item[•] the canonical projection $\pi _{\mathbb{R}^d}$ onto $\mathbb{R}^d$ is bijective from $\Sigma$ onto its image $\Gamma $,
\item[•] the image $\Gamma ^*$ of $\Sigma$ under the canonical projection $\pi _{\mathsf{H}}$ is a dense subgroup of $\mathsf{H}$.

\end{itemize}

\vspace{0.2cm}
Hence such an environment consists of an Euclidean space $\mathbb{R}^d$ embedded into $\mathsf{H}\times\mathbb{R}^{d}$ in an "irrational position" with respect to the lattice $\Sigma$. There is a well established formalism for these different ingredients: the space $\mathbb{R}^d$ is often called the \textit{physical space}, whereas the space $\mathsf{H}$ is called the \textit{internal space}. Moreover the morphism $\Gamma \longrightarrow \mathsf{H}$ which maps any $\gamma $ onto $\gamma ^*:= \pi _{\mathsf{H}}(\pi _{\mathbb{R}^d}^{-1}(\gamma ))\in \Gamma ^*$ is the \textit{*-map} of the cut and project scheme, whose graph is the lattice $\Sigma$. We will say that a cut and project scheme is \textit{real} whenever the internal space $\mathsf{H}$ is a finite dimensional real vector space $\mathbb{R}^n$.

\vspace{0.2cm}

We shall also consider a certain type of subset in the internal space $\mathsf{H}$, usually called a \textit{window}, which consists of a compact and topologically regular subset $W$, supposed \textit{irredundant} in the sense that the compact subgroup of elements of $w\in \mathsf{H}$ that satisfy $W+w = W$ is trivial. When $\mathsf{H}= \mathbb{R}^n$, this condition is immediately satisfied.

\vspace{0.2cm}

If we are given a cut and project scheme together with a window $W$ in its internal space, we may form a certain point pattern $\mathfrak{P}(W)$ in $\mathbb{R}^d$ by projecting onto the physical space the subset of points of the lattice $\Sigma$ lying within the \textit{strip} $W\times \mathbb{R}^d$. 
This is illustrated in Figure $1$ (see also \cite{BaaHePl}).

\vspace{0.2cm}
Figure $1$ presents the most simple real cut and project scheme one may consider, that is, with physical and internal spaces being $1$-dimensional, and with a lattice $\Sigma = \mathbb{Z}^2$ not crossing these spaces except at the origin. As window we consider the projection onto the internal space of the unit square in $\mathbb{R}^2$.

\includegraphics[trim = 4cm 18cm 5cm 4cm]{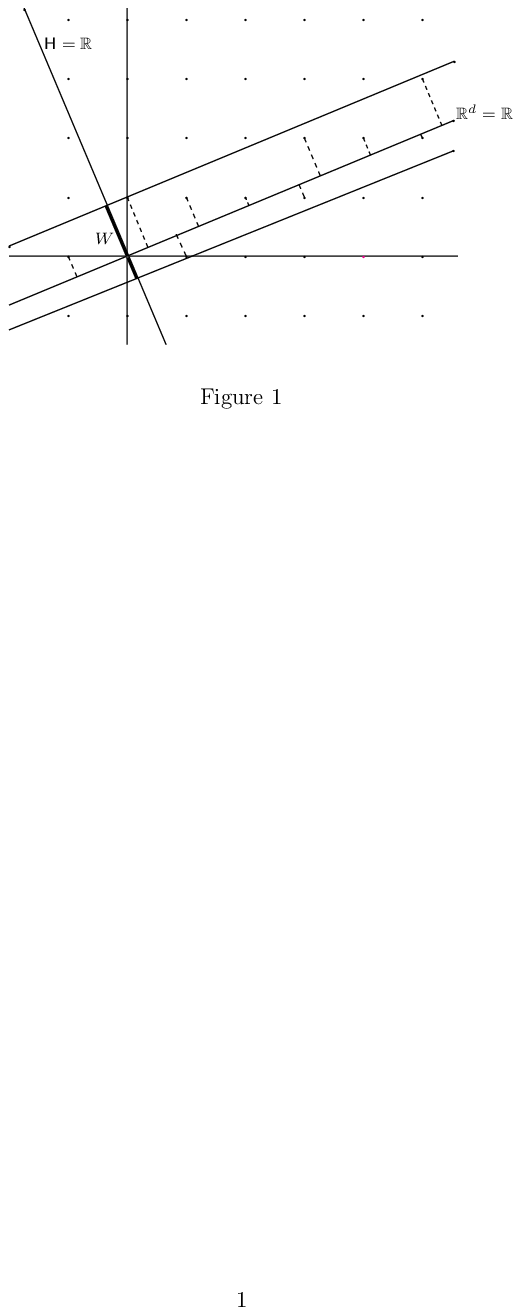}

The point pattern $\mathfrak{P}(W)$ may be written using the $*$-map as
\begin{align*} \mathfrak{P}(W) := \left\lbrace \gamma \in \Gamma \; \vert \; \gamma ^*\in W\right\rbrace . 
\end{align*}

We may allow ourselves to translate the resulting point pattern by any vector $t$ in the physical space $\mathbb{R}^d$, which we call here a \textit{physical translation}, or to translate the window $W$ by an element $w\in \mathsf{H}$, which we call an \textit{internal translation}. In both cases this leads to a new point pattern in $\mathbb{R}^d$. We now introduce the class of \textit{model sets} of $\mathbb{R}^d$ as follows:

\begin{de}\label{def:inter.model.set} An inter-model set $\Lambda$ associated to a cut and project scheme $(\mathsf{H}, \Sigma , \mathbb{R}^d)$ together with a window $W$ is a subset of $\mathbb{R}^d$ that satisfies
\begin{align*} \mathfrak{P}(w+ \stackrel{\circ }{W}) -t \subseteq \Lambda \subseteq \mathfrak{P}(w+ W) -t.
\end{align*}
\end{de}

An inter-model set is called \textit{regular} whenever the window $W$ used to construct it has boundary of Haar measure zero in $\mathsf{H}$. Due to the assumptions on the underlying cut and project scheme and on the window $W$, any inter model set is a \textit{Delone set}, that is a uniformly discrete and relatively dense subset of $\mathbb{R}^d$. In fact, it also satisfies the stronger property of being a \textit{Meyer set}, meaning that any inter-model set $\Lambda$ admits a uniformly discrete difference subset $\Lambda - \Lambda$ in $\mathbb{R}^d$. Most of the content of this article is about real cut and project schemes together with polytopic windows in their internal spaces, that hence provide regular inter-model sets.

\subsection{nonsingular model sets}

An important notion affiliated with a point pattern $\Lambda$ is its \textit{language}, namely the collection of all "circular-shaped" patterns appearing at sites of the point pattern:
\begin{align*}
\mathcal{L}_\Lambda := \left\lbrace (\Lambda -\gamma )\cap B(0,R) \, \vert \, \gamma \in \Lambda , \; R> 0\right\rbrace .
\end{align*}

Not all inter-model sets coming from a common cut and project scheme and window have same language. However, the class of \textit{nonsingular model sets}, also often called \textit{generic model sets}, do share a language:

\begin{de}\label{def:non.singular} A nonsingular model set is an inter-model set $\Lambda$ for which we have equalities
\begin{align*} \mathfrak{P}(w+ \stackrel{\circ }{W}) -t = \Lambda =\mathfrak{P}(w+ W) -t.
\end{align*}
\end{de}

The situation where such equality occurs for a given couple $(w,t)$ clearly only depends on the choice of $w\in \mathsf{H}$. We will then call an element $w\in \mathsf{H}$ nonsingular when the inter-model sets $\mathfrak{P}(w+ \stackrel{\circ }{W}) -t = \Lambda =\mathfrak{P}(w+ W) -t$ are nonsingular. Such a subset of nonsingular elements may easily be described: it consists of all $w\in \mathsf{H}$ where no point of the subgroup $\Gamma ^*$ of $\mathsf{H}$ enters the boundary $w+\partial W$ of the translated window $w+W$. It thus consists of the complementary subset
\begin{align*} NS:= \left[ \Gamma ^*-\partial W\right]  ^c .
\end{align*}
This set is always nonempty by the Baire category theorem, as $W$ was assumed topologically regular, hence with boundary of empty interior in $\mathsf{H}$, and $\Sigma$ (hence $\Gamma ^*$) was supposed to be countable. As already pointed out, the nonsingular model sets arising from a common cut and project scheme with window have a common language, which means that any pattern of some nonsingular model set appears elsewhere in all other nonsingular model sets. Denoting this language by $\mathcal{L}_{NS}$, we are led to consider its associated \textit{hull}.

\begin{de}\label{def:hull} Given a cut and project scheme and a window, and the language $\mathcal{L}_{NS}$ of any nonsingular model set arising from this data, the hull of this data is the collection 
\begin{align*} \mathbb{X}:= \left\lbrace  \Lambda \text{ point pattern of } \mathbb{R}^d \; \vert \mathcal{L}_\Lambda = \mathcal{L}_{NS} \right\rbrace  .
\end{align*}
We call a model set any point pattern within the hull $\mathbb{X}$ associated with some cut and project scheme and window.
\end{de}

The hull $\mathbb{X}$ associated with some cut and project scheme and window is also called the \textit{local isomorphism class} (or simply LI-class) of any model set within this hull.

\subsection{The hull as dynamical system}

There is a natural topology on the hull $\mathbb{X}$, which is metrizable and may be described by setting a basis of open neighborhoods of any point pattern $\Lambda \in \mathbb{X}$ to be (see \cite{Mo2})
\begin{align}\label{topology} \mathcal{U}_{K, \varepsilon }(\Lambda ):= \left\lbrace  \Lambda '\in \mathbb{X} \; \vert \; \exists \vert t \vert , \vert t'\vert < \varepsilon , \; (\Lambda -t)\cap K = (\Lambda '-t')\cap K \right\rbrace ,
\end{align}

where $K$ is any compact set in $\mathbb{R}^d$ and $\varepsilon >0$. This topology roughly means that two point patterns are close if they agree on a large domain about the origin up to small shifts. The hull $\mathbb{X}$, endowed with this topology, is a compact metrizable space, and is equipped with a natural action of $\mathbb{R}^d$ given by $\Lambda .t := \Lambda -t$, that is, by translating any point pattern. This provides a dynamical system $(\mathbb{X}, \mathbb{R}^d)$. To figure out what exactly this space consists of, we invoke the following beautiful result:

\begin{theo}\label{theo:parametrization.map} \cite{Sch} Let $\mathbb{X}$ be the hull associated with a cut and project scheme $(\mathsf{H},\Sigma,\mathbb{R}^d)$ and a window $W$. Then $\mathbb{X}$ is compact and the dynamical system $(\mathbb{X}, \mathbb{R}^d)$ is minimal. Each $\Lambda \in \mathbb{X}$ satisfies inclusions of the form
\begin{align*}
\mathfrak{P} (w_\Lambda +\mathring{W})-t_\Lambda \subseteq \Lambda \subseteq \mathfrak{P} (w_\Lambda +W)-t_\Lambda ,
\end{align*}
where $(w_\Lambda , t_\Lambda)\in \mathsf{H}\times \mathbb{R}^d$ is unique up to an element of $\Sigma$, thus defining a factor map
\begin{align*}
\pi : \mathbb{X}\longrightarrow  \mathsf{H}\times _{\Sigma} \mathbb{R}^d
\end{align*}

where $\mathsf{H}\times _{\Sigma} \mathbb{R}^d$ is the compact Abelian group quotient of $\mathsf{H}\times \mathbb{R}^d$ by the lattice $\Sigma$.\\ The map $\pi$ is injective precisely on the collection of nonsingular model sets of $\mathbb{X}$.
\end{theo}

By factor map we mean a continuous, onto and $\mathbb{R}^d$-equivariant map, where on the compact Abelian group $\mathsf{H}\times _{\Sigma} \mathbb{R}^d$ the space $\mathbb{R}^d$ acts through $[w,t]_{\Sigma}.s:= [w,t+s]_{\Sigma}$. In the context of real cut and project schemes the compact Abelian group is given by $\left[ \mathbb{R}^{n+d}\right] _{\Sigma} $, that is, it is a $(n+d)$-torus. In the theory of point patterns the above factor map is called the \textit{parametrization map}, and shows in particular that any model set of $\mathbb{X}$ is an inter-model set in the sense of Definition \ref{def:inter.model.set}. In fact, the collection $\mathbb{X}$ of model sets of a given cut and project scheme and window is precisely the collection of \textit{repetitive} inter-model sets arising from this data (see for instance \cite{Sch}).

\subsection{An explicit example}\label{subsection:explicit.example}

A well-known example of model set is the vertex point pattern of the famous \textit{Ammann-Beenker tiling}, from which an uncolored local pattern about the origin shows up as in Figure 2.

\includegraphics[ trim = 4cm 17cm 5cm 4.5cm]{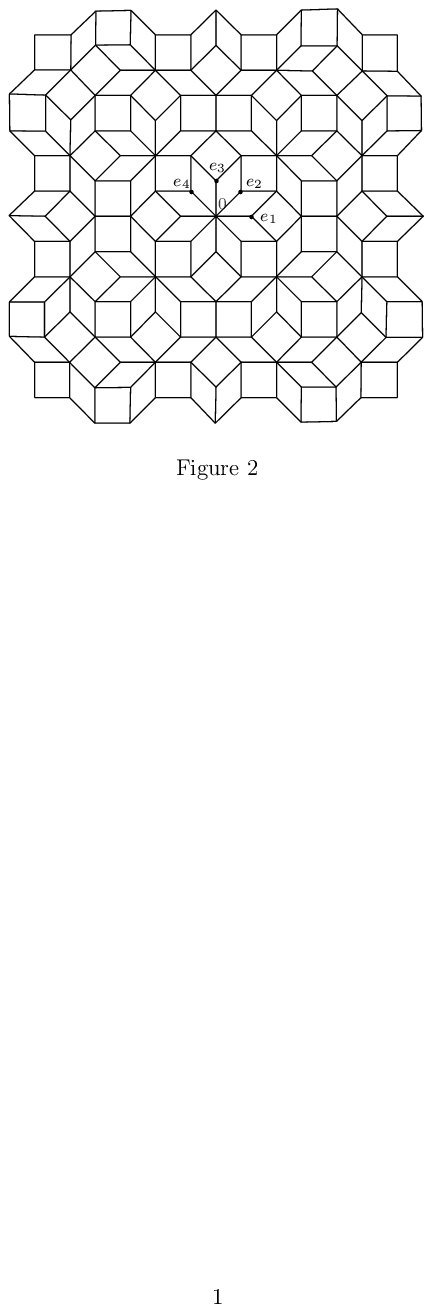}

We can set a cut $\&$ project scheme and window giving rise to the desired point pattern as follows: In a physical space $\mathbb{R}^2$ we embed the group $\Gamma $ algebraically generated by four vectors whose coordinates in an orthonormal basis are
\begin{align*} \begin{tabular}{c c c c} $e_1= (1,0),$ & $e_2 = (\frac{1}{\sqrt{2}},\frac{1}{\sqrt{2}}),$ & $e_3= (0,1),$ & $e_4 = (-\frac{1}{\sqrt{2}},\frac{1}{\sqrt{2}}).$
\end{tabular}
\end{align*} These four vectors are algebraically independent, and thus $\Gamma $ is isomorphic to $\mathbb{Z}^4$. Next we set the internal space $\mathbb{R}^2_{int}$ to be a $2$-dimensional real vector space, into which we define a $*-$map through the images of the four above vectors, reading in some orthonormal basis of $\mathbb{R}^2_{int}$ as
\begin{align*} \begin{tabular}{c c c c} $e_1^*= (1,0),\qquad e_2^* = (-\frac{1}{\sqrt{2}},\frac{1}{\sqrt{2}}),$ & $e_3^*= (0,-1),$ & $e_4^* = (\frac{1}{\sqrt{2}},\frac{1}{\sqrt{2}}),$
\end{tabular}
\end{align*}
The four vectors $\widetilde{e_i}:= (e_i,e_i^*)$, $i=1,2,3,4$, are linearly independant in $\mathbb{R}^2_{int}\times \mathbb{R}^2$ and thus form a lattice $\Sigma$, which projects onto a dense subgroup $\Gamma ^*$ in $\mathbb{R}^2_{int}$. This defines a real cut and project scheme. We choose the window to be canonical, that is, to be the projection to the internal space of the unit cube in $\mathbb{R}^2_{int}\times \mathbb{R}^2$ with respect to the basis $(\widetilde{e_1}, \widetilde{e_2}, \widetilde{e_3}, \widetilde{e_4})$. Hence we get a regular octagonal window $W_{oct}$ of the form in Figure 3.

Then the vertex point pattern appearing in the Ammann-Beenker tiling is given by the nonsingular model set $\mathfrak{P}\left( W_{oct} -\dfrac{e_1^*+e_2^*+e_3^*+e_4^*}{2}\right)$.



\includegraphics[ trim = 4cm 19cm 5cm 4.2cm]{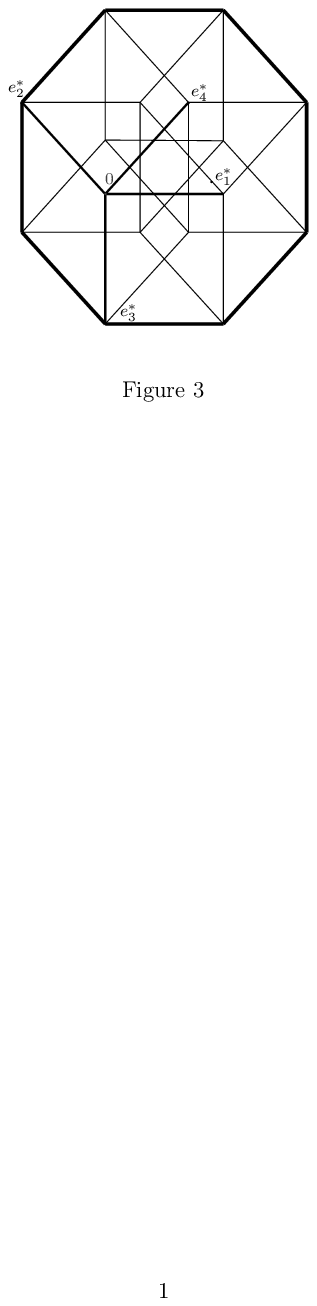}


\section{\large{\textbf{Ellis semigroups of dynamical systems}}}\label{section:ellis.semigroup}

\subsection{Ellis semigroup and equicontinuity}

Let us consider a compact dynamical system, that is, a compact (Hausdorff) space $\mathbb{X}$ together with an action of a group $T$ by homeomorphisms.

\begin{de}\label{def:ellis.semigroup} The Ellis semigroup $E(\mathbb{X},T)$ is the pointwise closure of the group of homeomorphisms given by the $T$-action in the space $\mathbb{X}^\mathbb{X}$ of self-mappings on $\mathbb{X}$.
\end{de}

The Ellis semigroup $E(\mathbb{X},T)$ is a family of transformations on the space $\mathbb{X}$ that are pointwise limits of homeomorphisms coming from the $T$-action, and is stable under composition. Moreover it is a compact (Hausdorff) space when endowed with the pointwise convergence topology coming from $\mathbb{X}^\mathbb{X}$. If the acting group is Abelian, then although the Ellis semigroup may not be itself Abelian, all of its transformations commute with any homeomorphism coming from the action.
\vspace{0.2cm}

The Ellis semigroup construction is functorial (covariant) in the sense that any onto continuous and $T$-equivariant mapping $\pi : \mathbb{X} \rightarrow \mathbb{Y}$ gives rise to an onto continuous semigroup morphism $\pi ^*: E(\mathbb{X},T) \rightarrow E(\mathbb{Y},T)$, satisfying $\pi (x.g)=\pi (x).\pi ^*(g)$ for any $x\in \mathbb{X}$ and any transformation $g\in E(\mathbb{X},T)$. Here we have written $x.g$ for the evaluation of a mapping $g$ at a point $x$. With this convention the Ellis semigroup is always a compact right-topological semigroup, that is, if some net $( h_\lambda )$ converges pointwise to $h$ then the net $( g.h_\lambda ) $ converges pointwise to $g.h$ for any $g$, where $g.h$ stands for the composition map which at each point $x$ reads $(x.g).h$.
\vspace{0.2cm}

Among the whole category of dynamical systems, the certainly most simple objects are the \textit{equicontinuous dynamical systems}. These are dynamical systems such that the family of homeomorphisms coming from the group action is equicontinous, and within the more specific class of compact minimal dynamical systems they exactly show up as the well-known class of \textit{Kronecker systems}. About these particular dynamical systems one has the following:

\begin{theo}\label{theo3} \cite{Aus} \cite{GlMeUs} Let $(\mathbb{X},T)$ be a minimal dynamical system over a compact metric space, with Abelian acting group. Then the following assertions are equivalent:
\vspace{0.1cm}
\begin{itemize}
\item[(1)] The dynamical system $(\mathbb{X},T)$ is equicontinuous.

\item[(2)] $E(\mathbb{X},T)$ is a compact group acting by homeomorphisms on $\mathbb{X}$.

\item[(3)] $E(\mathbb{X},T)$ is metrizable.

\item[(4)] $E(\mathbb{X},T)$ has left-continuous product.

\item[(5)] $E(\mathbb{X},T)$ is Abelian.

\item[(6)] $E(\mathbb{X},T)$ is made of continuous transformations.
\end{itemize}

\vspace{0.1cm}
In this case, $E(\mathbb{X},T)=\mathbb{X}$ as compact Abelian groups.
\end{theo}

Here the compact Abelian group structure of a compact minimal equicontinuous system $(\mathbb{X},T)$ with Abelian acting group is only determined by the choice (which is arbitrary) of one element $\mathfrak{e}\in \mathbb{X}$ which plays the role of a unit, from which the group structure extends that of $T$ mapped on the dense orbit $\mathfrak{e}.T$. In this case the equality $E(\mathbb{X},T)=\mathbb{X}$ is achieved by identifying a transformation $g\in E(\mathbb{X},T)$ with $\mathfrak{e}.g$ in $\mathbb{X}$.
\vspace{0.2cm}

Outside the scope of equicontinuous systems, the Ellis semigroup is a quite complicated object as it is formed of mappings neither necessarily continuous nor invertible, and is not commutative. However a general construction allows us to attach to any compact dynamical system a particular factor:

\begin{theo}\label{equ} Let $(\mathbb{X},T)$ be a compact dynamical system. There exist an equicontinuous dynamical system $(\mathbb{X}_{eq}, T)$ together with a factor map $\pi : \mathbb{X} \rightarrow \mathbb{X}_{eq}$ such that any equicontinuous factor of $(\mathbb{X},T)$ factors through $\pi$.
\end{theo}

The space $\mathbb{X}_{eq}$ with its $T$ -action is called the \textit{maximal equicontinuous factor} of $(\mathbb{X},T)$, and it is a Kronecker system whenever $(\mathbb{X},T)$ is topologically transitive. From Theorem \ref{theo3} one has $E(\mathbb{X}_{eq},T) = \mathbb{X}_{eq}$ as compact groups, and from the functorial property of the Ellis semigroup the quotient factor map $\pi$ from $\mathbb{X}$ onto its maximal equicontinuous factor gives rise to an onto and continuous semigroup morphism 
\begin{align*} \pi ^*:E(\mathbb{X},T)\longrightarrow \mathbb{X}_{eq}.
\end{align*}

\subsection{The tame property}

The following statement is obtained from \cite[Theorem 6.1, proof of Theorem 6.3]{GlMeUs}.

\begin{theo}\label{theo:dichotomy} \cite{GlMeUs} The Ellis semigroup $E(\mathbb{X},T)$ of a dynamical system over a compact metric space is either the sequential closure of the acting group $T$ in $\mathbb{X}^{\mathbb{X}}$ or contains a topological copy of the Stone-\v{C}ech compactification $\beta \mathbb{N}$ of the integers.
\end{theo}

The first case of this dichotomy admits several different formulations (see Glasner \cite{Gl} and Glasner, Megrelishvili and Uspenskij \cite{GlMeUs} and references therein) and whenever it occurs the underlying dynamical system is called \textit{tame}. If a compact metric dynamical system admits an Ellis semigroup with first countable topology then it is automatically a tame system. The tame property is related to the following notion:

\begin{de} A compact dynamical system $(\mathbb{X},T)$ is almost automorphic if the factor map $\pi : \mathbb{X} \rightarrow \mathbb{X}_{eq}$ possesses a one-point fiber.
\end{de}

\begin{theo}\label{theo:glasner} \cite{Gl2} If a compact metric minimal dynamical system with Abelian acting group is tame, then it is almost automorphic.
\end{theo}


In case of metrizability of the space $\mathbb{X}$ Veech showed in \cite[Lemma 4.1]{Vee} that any almost automorphic system in fact has a residual subset of one-point fibers with respect to the mapping $\pi$. In the situation of a hull $\mathbb{X}$ of model sets, the factor map $\pi$ onto the maximal equicontinuous factor is precisely the parametrization map of Theorem \ref{theo:parametrization.map}, and thus $(\mathbb{X}, \mathbb{R}^d)$ is almost automorphic since $\pi$ is one-to-one on a nonempty subclass of $\mathbb{X}$ (the nonsingular model sets). Even more is true: The hull $\mathbb{X}$ consists of regular model sets (meaning that the region $W$ as its boundary of null Haar measure in $\mathsf{H}$) if and only if the map $\pi$ is one-to-one above a full haar measure subset of $\mathbb{X}_{eq}$ \cite[Theorem 5]{BaaLenMo}.

\subsection{Ellis semigroup for locally compact dynamical systems}

We wish to include here two elementary results about the Ellis semigroup one may define for dynamical systems over locally compact spaces. Let $\mathbb{X}$ be a locally compact space together with an action of a group $T$ by homeomorphisms, and as in the compact case, set the Ellis semigroup $E(\mathbb{X},T)$ to be the pointwise closure in the product space $\mathbb{X}^\mathbb{X}$ of the group of homeomorphisms coming from the $T$--action. In order to extend some results available in the compact case to this setting we consider the one-point compactification $\hat{\mathbb{X}}$ of $\mathbb{X}$, endowed with a $T$--action by homeomorphisms so that the infinite point remains fixed through any such homeomorphism. Let us denote by $\mathcal{F}_{\mathbb{X}}$ the subset of $\hat{\mathbb{X}}^{\hat{\mathbb{X}}}$ of transformations mapping $\mathbb{X}$ into itself and keep the point at infinity fixed, endowed with relative topology. Then $\mathcal{F}_{\mathbb{X}}$ is a semigroup which is isomorphic and homeomorphic with the product space $\mathbb{X}^\mathbb{X}$, and under this identification
\begin{align*} E(\mathbb{X},T)= E(\hat{\mathbb{X}},T)\cap \mathcal{F}_{\mathbb{X}}.
\end{align*}
Observe that $E(\mathbb{X},T)$ is, as in the compact flow case, a right-topological semigroup containing $T$ as a dense subgroup (or rather the subsequent group of homeomorphisms). The following is a general fact, whose proof for compact dynamical systems can be found in \cite{Aus}:

\begin{prop}\label{prop:map.ellis.loc.compact} Let $\pi : \mathbb{X}\longrightarrow \mathbb{Y}$ be a continuous, proper, onto, and $T$--equivariant map between locally compact spaces. Then there exists a continuous, proper, and onto morphism $\pi ^*: E(\mathbb{X},T) \longrightarrow E(\mathbb{Y},T)$ that satisfies the equivariance condition $\pi (x.g)= \pi (x).\pi ^*(g)$ for any $x\in \mathbb{X}$ and $g\in E(\mathbb{X},T)$.
\end{prop}

\begin{proof} Denote by $\star _\mathbb{X} $ and $\star _\mathbb{Y}$ the respective points at infinity in the compactified spaces. Since $\pi $ is continuous and proper, it extends to a continuous and onto map $\widehat{\pi }: \widehat{\mathbb{X}}\rightarrow \widehat{\mathbb{Y}}$, such that $\widehat{\pi }^{-1}(\star _\mathbb{Y})= \lbrace \star _\mathbb{X} \rbrace$. Obviously $\widehat{\pi }$ is $T$--equivariant with respect to the extended $T$--actions. Then there exists a continuous and onto morphism $\widehat{\pi }^*: E(\widehat{\mathbb{X}},T) \rightarrow E(\widehat{\mathbb{Y}},T)$, satisfying the equivariance equality $\widehat{\pi }(x.g)= \widehat{\pi }(x).\widehat{\pi }^*(g)$ for any $x\in \widehat{\mathbb{X}}$ and $g\in E(\widehat{\mathbb{X}},T)$.
The later equivariance condition implies that a transformation $g$ of $E(\widehat{\mathbb{X}},T)$ lies in $\mathcal{F}_{\mathbb{X}}$ if and only if $\widehat{\pi }^*(g)$ lies in $\mathcal{F}_{\mathbb{X}}$: it follows that $E(\mathbb{X},T)= (\widehat{\pi }^*)^{-1}(E(\mathbb{Y},T))$. Restricting the morphism onto $E(\mathbb{X},T)$ gives the map, together with the onto property.
Finally a compact set of $E(\mathbb{Y},T)$ must be compact in $E(\widehat{\mathbb{Y}},T)$, as it is easy to check, so have a compact inverse image in $E(\widehat{\mathbb{X}},T)$ under $\widehat{\pi }^*$. This latter is entirely included in $E(\mathbb{X},T)$, so is compact for the relative topology on $E(\mathbb{X},T)$. This gives the properness.
\end{proof}

Observe that $\pi ^*(t)=t$ holds for any $t\in T$.
As in the compact setting, if the acting group $T$ is Abelian then any induced homeomorphism commutes with any mapping in $E(\mathbb{X},T)$. To end this subsection we state without proof an easy property of locally compact Kronecker systems:

\begin{prop}\label{prop:equi.loc.compact} If $T$ is a dense subgroup of a locally compact Abelian group $\mathbb{G}$, $T$ acting by translation, then $E(\mathbb{\mathbb{G}},T)$ is topologically isomorphic with $\mathbb{G}$, where any $g\in \mathbb{G}$ is identified with its translation map in $E(\mathbb{\mathbb{G}},T)$.
\end{prop}

\section{\large{\textbf{The internal system of a hull of model sets}}}\label{internal system}

\subsection{Internal system}

What we introduce here is an analogue in the hull $\mathbb{X}$ of the internal space $\mathsf{H}$ we may find in the compact Abelian group $\mathsf{H}\times _{\Sigma} \mathbb{R}^d$ (the torus $\left[ \mathbb{R}^{n+d}\right] _{\Sigma}$ in case of a real cut $\&$ project scheme). We call this analogue the \textit{internal system} of a hull of model sets. The consideration of this particular space is not new (it appeared in \cite{FoHuKe} as well as in the formalism of $C^*$--algebras in \cite{Put}), although it is often not explicitly mentioned, and we record here the main aspects of this space.

\begin{de}\label{def:internal.system} Let $\mathbb{X}$ be the hull of model sets associated with a cut $\&$ project scheme $(\mathsf{H},\Sigma,\mathbb{R}^d)$ and a window $W$. Then its internal system is the subclass $\Xi ^\Gamma$ of point patterns that are supported on the structure group $\Gamma$ in $\mathbb{R}^d$, that is,
\begin{align*} \Xi ^\Gamma := \left\lbrace \Lambda \in \mathbb{X} \, \vert \; \Lambda \subset \Gamma \right\rbrace  .
\end{align*}
\end{de}

According to Theorem \ref{theo:parametrization.map}, any model set admits inclusions of the form stated in Definition \ref{def:inter.model.set}, and we can see that the subclass $\Xi ^\Gamma$ consists exactly of the model sets for which these inclusions read
\begin{align*} \mathfrak{P}(w+ \stackrel{\circ }{W}) \subseteq \Lambda \subseteq \mathfrak{P}(w+ W).
\end{align*}
Equivalently, $\Xi ^\Gamma$ is the subclass of model sets in $\mathbb{X}$ whose image under the parametrization map $\pi$ of Theorem \ref{theo:parametrization.map} is of the form $[w,0]_{\Sigma}$ in the compact Abelian group $\mathsf{H}\times _{\Sigma} \mathbb{R}^d$. On the other hand, there exists a natural morphism mapping any element $w$ of the internal space $\mathsf{H}$ to $[w,0]_{\Sigma}$ in $\mathsf{H}\times _{\Sigma} \mathbb{R}^d$, which is one-to-one and continuous. This suggests the existence of a mapping from the internal system $\Xi ^\Gamma$ onto the internal space $\mathsf{H}$ of the cut $\&$ project scheme. However, similar to the fact that $\mathsf{H}$ is in general not topologically conjugated with its image in $\mathsf{H}\times _{\Sigma}\mathbb{R}^d$, one cannot just consider the topology of $\mathbb{X}$ restricted on $\Xi ^\Gamma$. Rather, we consider on the internal system the topology whose basis of open neighborhoods around any $\Lambda \in \Xi ^\Gamma$ is
\begin{align}\label{topology.internal}\mathcal{U}_K(\Lambda):= \left\lbrace \Lambda '\in \Xi ^\Gamma \; \vert \;  \Lambda \cap K = \Lambda '\cap K \right\rbrace ,
\end{align}
where $K$ is any compact set in $\mathbb{R}^d$. This means that two point patterns are close in $\Xi ^\Gamma$ if they exactly match on a large domain about the origin. On the internal system equipped with the above topology, we consider the action of the group $\Gamma$ by homeomorphisms given by translation on each model set, so that one obtains a dynamical system $(\Xi ^\Gamma , \Gamma)$. From the minimality of the dynamical system $(\mathbb{X},\mathbb{R}^d)$ by Theorem \ref{theo:parametrization.map}, this dynamical system is also minimal. Of particular importance is the sub-collection, often called the \textit{transversal}, of point patterns containing the origin
\begin{align*} \Xi := \left\lbrace \Lambda \in \mathbb{X} \, \vert \; 0\in \Lambda  \subset \Gamma \right\rbrace .
\end{align*}
Any model set containing the origin must be entirely included in the structure group $\Gamma$, that is, $\Xi $ is a subset of the internal system $\Xi ^\Gamma$. A fact of fundamental importance is that $\Xi$ is in fact a \textit{clopen set}, that is, a subset which is both open and closed in the internal system $\Xi ^\Gamma$: Indeed any accumulation point pattern of $\Xi$ must possess the origin in its support, and thus is actually an element of $\Xi$, and on the other hand for each $\Lambda \in \Xi$ and any radius $R>0$ the collection $ \mathcal{U}_{B(0,R)}(\Lambda )$ is an open neighborhood of $\Lambda$ in $\Xi ^\Gamma$ that is clearly contained in the transversal $\Xi$.

\vspace{0.2cm}
About the topology of the transversal one may observe that there is a one to one correspondence between circular-shaped local configuration of radius $R$ in the language $\mathcal{L}_{NS}$ and subsets in $\Xi$ of the form $ \mathcal{U}_{B(0,R)}(\Lambda )$, for the same radius $R$ and $\Lambda$ chosen in $\Xi$. Thus compactness of $\Xi$ is rephrased by the existence of only a finite number of such circular-shaped patterns for any radius $R$, a property called \textit{finite local complexity} for the underlying point patterns in $\mathbb{X}$. This property holds in our context \cite{Mo2, Sch}, so that $\Xi$ is a compact open subset of the internal system $\Xi ^\Gamma$.

\begin{prop}\label{prop:topology.internal.system.tot.disconnected} The internal system $\Xi ^\Gamma$ of a hull of model sets is a totally disconnected locally compact topological space, and a subbasis for its topology is formed by all $\Gamma$--translates of $\Xi$ and its complementary set $\Xi ^c$.
\end{prop}

\begin{proof} Any point pattern $\Lambda \in \Xi ^\Gamma$ is uniquely determined by the knowledge of whether a point $\gamma \in \Gamma$ lies in $\Lambda $ or not, for each $\gamma \in \Gamma$. Thus a sub-basis for the topology of the internal system is given by the subsets
\begin{align*} \Xi _\gamma:= \left\lbrace \Lambda \in \Xi ^\Gamma \, \vert \; \gamma\in \Lambda  \right\rbrace 
\end{align*}
or their complementary sets $\Xi _\gamma ^c$ for $\gamma \in \Gamma$. Since they are both open they are thus both closed as well, giving that the internal system is totally disconnected. Now any $\Xi _\gamma$ or $\Xi _\gamma ^c$ is nothing but the $-\gamma$ translate of $\Xi$ or the complementary set $\Xi ^c$, as $$\Lambda \in \Xi .(-\gamma ) \Longleftrightarrow \Lambda .\gamma \in \Xi \Longleftrightarrow 0\in (\Lambda -\gamma) \Longleftrightarrow \gamma \in \Lambda \Longleftrightarrow \Lambda \in \Xi _\gamma$$
As any point pattern $\Lambda \in \Xi ^\Gamma$ must contain at least one element $\gamma \in \Gamma$, one gets that the compact open subsets $\Xi _\gamma$ form a covering of $\Xi^\Gamma$, giving the local compactness.
\end{proof}

\begin{prop}\label{prop:parametrization.map.interne}\cite{Sch} Let $\Xi ^\Gamma$ be the internal system associated with a cut $\&$ project scheme $(\mathsf{H},\Sigma,\mathbb{R}^d)$ and some window. Then there exists a factor map
\begin{align*}
\Pi : \Xi ^\Gamma \longrightarrow  \mathsf{H} 
\end{align*}

mapping a point pattern $\Lambda$ onto the unique element $\Pi (\Lambda )= w_\Lambda $ of $\mathsf{H}$ satisfying
\begin{align*}
\mathfrak{P} (w_\Lambda +\mathring{W}) \subseteq \Lambda \subseteq \mathfrak{P} (w_\Lambda +W).
\end{align*}
Moreover, the map $\Pi$ satisfies $\Pi(\Xi)= -W$, and is injective precisely on the dense family of nonsingular model sets of $\Xi ^\Gamma$, whose image is the dense subset $NS$ of $\mathsf{H}$.
\end{prop}

From the above proposition we thus have a correspondence between any $w\in NS$ with a unique nonsingular model set $\mathfrak{P}(w+ W)\in \Xi ^\Gamma$, and we may also write $NS$ for the dense subclass of nonsingular model sets of $\Xi ^\Gamma$. Thus the internal system $\Xi ^\Gamma$ and the internal space $\mathsf{H}$ as different completions of a single set $NS$. This observation allows us to set, for any subset $A$ of $\mathsf{H}$, a corresponding subset of $\Xi ^\Gamma $ of the form
\begin{align}\label{coupe} [A]_{\Xi}:= \overline{A\cap NS}^{\Xi ^\Gamma}
\end{align}
Such a $[A]_{\Xi}$ will be non-empty if and only if $A$ intersects $NS$. In particular $[A]_{\Xi}$ will have non-empty interior (and hence will be non-empty) whenever $A$ has non-empty interior. We will have use of the following lemma, which we state without proof:

\begin{lem}\label{lem:clopen} Let $X$ be a topological space, and $Y$ a dense subset. Then each clopen subset $V$ of $X$ is equal to the closure of $V\cap Y$. Moreover, if two clopen subsets coincide on $Y$, then they are equal.
\end{lem}

For instance, one is able to show $\Xi = [-W]_\Xi = [-\mathring{W}]_\Xi$, underlying a link between the topology of the internal system and the geometry of $W$ in the internal space.

\begin{prop}\label{prop:internal.morphism.ellis} There exists an onto and proper continuous morphism
\begin{align*}
\Pi ^* :  E(\Xi ^\Gamma,\Gamma)\longrightarrow  \mathsf{H} 
\end{align*}
that satisfies the equivariance relation $\Pi(\Lambda .g)= \Pi(\Lambda)+\Pi^*(g)$ for any model set $\Lambda \in \Xi ^\Gamma$ and any mapping $g\in E(\Xi ^\Gamma,\Gamma)$. 
\end{prop}

\begin{proof} Let us show that the the map $\Pi $ of Proposition \ref{prop:parametrization.map.interne} is proper, that is, the inverse image of any compact set of $\mathsf{H}$ is compact in $\Xi ^\Gamma$: Let $K$ be a compact subset of $\mathsf{H}$ and pick a model set $\Lambda$ in $\Pi ^{-1}(K)$. Since $\Gamma ^*$ is dense in $\mathsf{H}$ there exist $\gamma _1, ..., \gamma _l$ in $\Gamma$ such that $$K\subset \bigcup _{k=1 }^l \gamma ^* _k-\mathring{W}.$$ Thus $w_\Lambda \in K$ falls into some $\gamma ^* _k-\mathring{W}$, which implies that $\gamma _k$ lies in $ \mathfrak{P}(w_\Lambda +\mathring{W}) \subset \Lambda$. This in turns means that $0\in \Lambda -\gamma _k$, so $\Lambda -\gamma _k \in \Xi$ and thus $\Lambda \in \Xi .(-\gamma _k)$. Hence the closed set $\Pi ^{-1}(K)$ is entirely included in a finite union of translates of $\Xi$, each being compact, and so is a compact set of $\Xi ^\Gamma $. Now Proposition \ref{prop:map.ellis.loc.compact} applies, and after invoking Proposition \ref{prop:equi.loc.compact} it gives the desired morphism $\Pi ^*$.
\end{proof}

\subsection{Hull and internal system Ellis semigroups}

We want to relate the Ellis semigroup of the dynamical systems $(\mathbb{X},\mathbb{R}^d)$ with that of $(\Xi ^\Gamma, \Gamma )$. To this end, let $g$ be any mapping in the Ellis semigroup $E(\Xi ^\Gamma , \Gamma )$. Using Theorem \ref{theo:parametrization.map} together with the definition of the internal system, one sees that any point pattern $\Lambda $ in $\mathbb{X}$ can be written as $\Lambda _0= \Lambda -t$ for some model set $\Lambda \in \Xi ^\Gamma $ and some vector $t\in \mathbb{R}^d$. The mapping $g$ is well defined on each $\Lambda \in \Xi ^\Gamma$, and we may thus extend it into a self-map $\widetilde{g}$ of $\mathbb{X}$ by setting
\begin{align}\label{extension}\Lambda _0 .\widetilde{g} =(\Lambda -t).\widetilde{g} := \Lambda .g -t.
\end{align}
This is well defined since if one has $\Lambda -t=\Lambda '-t'$ with $\Lambda$ and $\Lambda '\in \Xi ^\Gamma$ then necessarily $\Gamma -t= \Gamma -t'$, which means that $t-t'\in \Gamma$, and since $g$ commutes with the $\Gamma$--action on $\Xi ^\Gamma$ then applying (\ref{extension}) gives the same result. Let us now consider the semigroup $E(\Xi ^\Gamma , \Gamma )\times _\Gamma \mathbb{R}^d $ to be the (topological) quotient of the direct product semigroup $E(\Xi ^\Gamma , \Gamma )\times  \mathbb{R}^d$ by the normal subsemigroup formed of elements $(\gamma , \gamma)$ with $\gamma \in \Gamma$.

\begin{theo}\label{theo:suspension.ellis} Let $\mathbb{X}$ and $\Xi ^\Gamma$ be the hull and internal system generated by a cut $\&$ project scheme $(\mathsf{H} ,\Sigma,\mathbb{R}^d)$ and some window. Then there is a homeomorphism and semigroup isomorphism
\begin{align*}  E(\Xi ^\Gamma , \Gamma )\times _\Gamma \mathbb{R}^d \simeq E(\mathbb{X}, \mathbb{R}^d)
\end{align*}
mapping each element $[g, t]_\Gamma $ of $E(\Xi ^\Gamma , \Gamma )\times _\Gamma \mathbb{R}^d $ to $\widetilde{g}-t$.
\end{theo}

\begin{proof} First we show that the quotient semigroup $E(\Xi ^\Gamma , \Gamma )\times _\Gamma \mathbb{R}^d $ is compact (Hausdorff): From the existence of the morphism $\Pi ^*$ one then gets a natural onto semigroup morphism $\Pi ^*\times id: E(\Xi ^\Gamma, \Gamma)\times \mathbb{R}^d \rightarrow \mathsf{H} \times \mathbb{R}^d$, which maps the normal subsemigroup formed by elements $(\gamma , \gamma)$ with $\gamma \in \Gamma$ onto the lattice $\Sigma$. Since $\mathsf{H} \times _{\Sigma}\mathbb{R}^d$ is compact (Hausdorff), and since $\Pi ^*\times id$ is continuous and proper, we deduce that $E(\Xi ^\Gamma , \Gamma )\times _{\Gamma}\mathbb{R}^d$ must itself be compact (Hausdorff).
\vspace{0.2cm}

Now it is clear that the mapping associating $g\in E(\Xi ^\Gamma , \Gamma )$ with $\tilde{g}\in \mathbb{X}^\mathbb{X}$ is a semigroup morphism for the composition laws of mappings. This association is moreover continuous: For this it suffices to check the continuity of each evaluation map $g\mapsto \Lambda _0. \tilde{g}$, with $\Lambda _0\in \mathbb{X}$. Write $\Lambda _0$ as $\Lambda  - t$ for some $\Lambda \in \Xi ^\Gamma$. If $g_\lambda$ is a net in $E(\Xi ^\Gamma , \Gamma )$ pointwise converging in $\Xi ^\Gamma$ to a mapping $g$, then one has convergence of the net $\Lambda .g_\lambda$ to $\Lambda  .g$ in the internal system $\Xi ^\Gamma$. Comparing the topologies coming from the topologies (\ref{topology}) on $\mathbb{X}$ and (\ref{topology.internal}) on $\Xi ^\Gamma$, one sees that the embedding of $\Xi ^\Gamma$ into the hull $\mathbb{X}$ is continuous, so that the net $\Lambda .g_\lambda$ also converges to $\Lambda  .g$ in the hull $\mathbb{X}$. Hence the net $\Lambda _0. \tilde{g}_\lambda = \Lambda  . g_\lambda -t$ converges to $\Lambda .g -t= \Lambda _0. \tilde{g}$, as desired.
\vspace{0.2cm}

From this we can define a continuous semigroup morphism from $E(\Xi ^\Gamma , \Gamma )\times  \mathbb{R}^d $ into $\mathbb{X}^\mathbb{X}$ that associates a pair $(g, t)$ with $\tilde{g}-t$. Clearly any pair of the form $(\gamma , \gamma)$ with $\gamma \in \Gamma$ is mapped to the identity map, thus giving a continuous semigroup morphism
\begin{align*}
E(\Xi ^\Gamma , \Gamma )\times _\Gamma \mathbb{R}^d \longrightarrow \mathbb{X}^\mathbb{X}, \quad  \left[ g, t\right] _\Gamma \longmapsto \tilde{g}-t .
\end{align*}
This map is one-to-one: If $\left[ g, t\right] _\Gamma$ and $\left[ g', t'\right] _\Gamma$ are such that $\tilde{g}-t \equiv \tilde{g}'-t' $, then they must in particular coincide at any model set $\Lambda \in \Xi ^\Gamma$, thus giving for each such point pattern $\Lambda .g -t = \Lambda .g' -t'$. As $\Lambda .g$ and $\Lambda .g'$ are supported on $\Gamma $ we deduce that $t'-t=: \gamma\in \Gamma$, and that $g'$ coincides with $g+\gamma$ everywhere on $\Xi ^\Gamma$. It follows that $\left[ g', t'\right] _\Gamma = \left[ g + \gamma, t + \gamma\right] _\Gamma = \left[ g, t\right] _\Gamma$, giving injectivity. Now the stated morphism conjugates, both topologically and algebraically, the semigroup $E(\Xi ^\Gamma , \Gamma )\times  \mathbb{R}^d $ with its image in $\mathbb{X}^\mathbb{X}$. To conclude it suffices then to show that this image densely contains the group of homeomorphisms coming from the $\mathbb{R}^d$--action on $\mathbb{X}$. Obviously this group is contained into the image in question, appearing as $[0 ,t]_\Gamma $, where $0$ stands for the identity mapping on $\Xi ^\Gamma$, lying in $\Gamma$ and thus in $E(\Xi ^\Gamma , \Gamma )$. Let $\tilde{g}-t$ be some mapping in this image. A neighborhood basis for this latter in $\mathbb{X}^\mathbb{X}$ may be given as finite intersections of sets
\begin{align*} V_{\mathbb{X}}(\Lambda , U):= \lbrace f\in \mathbb{X}^\mathbb{X} \, \vert \, \Lambda .f \in U\rbrace 
\end{align*} containing $\tilde{g}-t$. Let $\Lambda _1 , .., \Lambda _k$ be model sets and $U_1,..,U_k$ be open subsets of $\mathbb{X}$ such that $\tilde{g}-t$ lies in $V(\Lambda _j,U_j)$ for each $j$. To get density it then suffices to show the existence of some element of $\mathbb{R}^d$ also contained in $V(\Lambda _j,U_j)$ for each $j$. Let us write $\Lambda _j$ as a sum $\Lambda _j' - t_j$ with $\Lambda _j '\in \Xi ^\Gamma$. Hence the mapping $g$, being the restriction of $\tilde{g}$ on $\Xi ^\Gamma$, lies in each subset 
\begin{align}\label{open.sets} V_{\Xi ^\Gamma}(\Lambda _j', \Xi ^\Gamma \cap (U_j+t +t_j)):= \left\lbrace f\in (\Xi ^\Gamma )^{\Xi ^\Gamma} \, \vert \, \Lambda _j'.f \in U_j+t +t_j\right\rbrace 
\end{align}
The embedding of $\Xi ^\Gamma$ in the hull $\mathbb{X}$ is clearly continuous, so that $\Xi ^\Gamma \cap (U_j+t +t_j)$ are open sets of the internal system, so the sets (\ref{open.sets}) are open in $(\Xi ^\Gamma ) ^{\Xi ^\Gamma}$. As $E(\Xi ^\Gamma , \Gamma )$ is the closure of $\Gamma $ in $(\Xi ^\Gamma ) ^{\Xi ^\Gamma}$ one may thus find some $\gamma \in \Gamma$ within each set (\ref{open.sets}), giving that $\gamma -t\in \mathbb{R}^d$ lies in each $V(\Lambda _j,U_j)$, as desired. \end{proof}


Apart from this, the parametrization map $\pi$ of Theorem \ref{theo:parametrization.map} also implies the existence of an onto continuous semigroup morphism
\begin{align*}
\pi ^* : E(\mathbb{X},\mathbb{R}^d)\longrightarrow  \mathsf{H}\times _{\Sigma}\mathbb{R}^d 
\end{align*}
that satisfies the equivariance relation $\pi(\Lambda .\mathsf{g})= \pi(\Lambda)+\pi^*(\mathsf{g})$ for any model set $\Lambda \in \mathbb{X}$ and any Ellis transformation $\mathsf{g}\in E(\mathbb{X},\mathbb{R}^d)$. Then the morphism $\pi ^*$ extends the morphism $\Pi ^*$ in the sense that for any transformation $g$ in $E(\Xi ^\Gamma , \Gamma )$ and $t\in \mathbb{R}^d$, one has the equality
\begin{align*}\pi ^*(\tilde{g}-t)=[\Pi ^*(g),t]_{\Sigma}.
\end{align*}

\section{\large{\textbf{The almost canonical property for model sets}}}\label{almost canonical property}

We wish to define here the almost canonical property on a model set. To this end we restrict ourselves to real cut and project scheme $(\mathbb{R}^n, \Sigma , \mathbb{R}^d)$, and we ask the window $W$ to be a $n$--dimensional \textit{compact convex polytope} of the internal space $\mathbb{R}^n$. The definition of almost canonical model sets will be derived from a corresponding notion on $W$, which consists of a pair of assumptions we will now present.
\vspace{0.2cm}

In fact, it will be much more convenient to consider the \textit{reversed window} $M:= -W$ in the internal space. It as well consists of a $n$--dimensional compact convex polytope in $\mathbb{R}^n$, whose boundary is given by $\partial M = -\partial W$. If we now let $f$ be any $(n-1)$ dimensional face of $M$, we define:

\begin{itemize}

\item $A_f$ or $A_f^0$ to be the affine hyperplane generated by $f$.

\item $H_f$ or $H_f^0$ be the corresponding linear hyperplane in $\mathbb{R}^n$.

\item $Stab_{\Gamma}(A_f)$ to be the subgroup of $\gamma \in \Gamma$ with $\gamma ^*\in H_f$.
\end{itemize}

We remark that $Stab_{\Gamma}(A_f)$ is precisely the subgroup of elements $\gamma \in \Gamma$ such that $A_f+\gamma ^*=A_f$, whence the notation. We may also denote $Stab_{\Gamma}(A_f)^*$ for its image in the internal space under the $*$--map.

\begin{as}\label{as1} For each face $f$ of $M$, the sum $Stab_{\Gamma}(A_f)^*+f$ covers $A_f$ in $\mathbb{R}^n$.
\end{as}

The above assumption implies in particular that $Stab_{\Gamma}(A_f)$ has a relatively dense image in $H_f$ under the $*$--map, and thus must be of rank at least $n-1$. Under the above assumption we get a nice description of the subset of nonsingular vectors:
\begin{align*} NS:= \left[ \Gamma ^*-\partial W\right]  ^c = \left[ \Gamma ^*+\partial M\right]  ^c =  \left[ \bigcup _{ f \text{ face of } M} \Gamma^* +A_f\right] ^c .
\end{align*}
As we see above, the subset of nonsingular vectors arises as the complementary subsets of all the $\Gamma$--translates of \textit{singular hyperplanes}, namely the affine hyperplanes $A_f$ with $f$ a face of the reversed window $M$. Let us in addition define for each $(n-1)$--dimensional face $f$ of $M$:

\begin{itemize}

\item $H_f ^-$ and $H_f^+$ to be the open half-spaces with boundary $H_f$.

\item $H_f ^{-0}$ and $H_f^{+0}$ to be the closed half-spaces with boundary $H_f$.

\item $A_f ^-$, $A_f^+$, $A_f ^{-0}$ and $A_f^{+0}$ be the corresponding objects with respect to $A_f$.
\end{itemize}

The choice of orientation on each linear hyperplane provided by the above notation is not relevant, but will be remained fixed from now on. Observe that a hyperplane $H$ may be associated to two different faces, which in this case leaves a common orientation on the corresponding affine spaces.
\vspace{0.2cm}

Recall that to any Euclidean subset $A$ may be associated a corresponding subset $[A]_\Xi$ of the internal system according to (\ref{coupe}). We will be specially interested here in a certain collection of Euclidean subsets which we call the family of \textit{admissible half-spaces},
\begin{align*} \mathfrak{A}= \left\lbrace \gamma ^*+ A_f^\pm \, \vert \, \gamma \in \Gamma , \, f \text{ face of } M \right\rbrace 
\end{align*} 

\begin{as}\label{as2} Any set $[A]_{\Xi}$, where $ A\in \mathfrak{A}$, is a clopen set.
\end{as}

It can be shown that Assumption $2$ implies Assumption $1$, but as we don't really need to prove this fact here we assume both independently. We wish to illustrate what type of polytope could satisfies Assumptions $1$ and $2$ by showing situations where this holds, but first let us define what an \textit{almost canonical model set} is:

\begin{de}\label{def:almost.canonical.model.set} A model set is almost canonical when it may be constructed with a real cut and project scheme and a compact convex polytopic window in its internal space that satisfy Assumptions $1$ and $2$.
\end{de}

The term \textit{almost canonical} makes reference to the first point patterns defined as model sets, the \textit{canonical} model sets, constructed via a real cut and project scheme $(\mathbb{R}^n, \Sigma , \mathbb{R}^d)$ together with an orthogonal projection of a unit cube for a window, with respect to the lattice $\Sigma$, in the internal space. Our example in Section $1$ is of this form. The terminology \textit{almost canonical} has been introduced by Julien in \cite{Ju} to define slight generalizations of these model sets. However, our definition doesn't fit exactly with the one given in \cite{Ju} (it can be shown that ours implies the one of Julien), but as we don't want to introduce another definition for something so close to the situation of \cite{Ju}, we allow ourselves to abuse terminology and call ours almost canonical. As shown in \cite{FoHuKe}, a canonical window always satisfies Assumptions $1$ and $2$ and is thus almost canonical in our sense.

\vspace{0.2cm}
A condition that ensures that Assumptions $1$ and $2$ hold is the requirement that any stabilizer $Stab_\Gamma (A_f)$ is dense in the corresponding linear hyperplane $H_f$ for any face $f$ of the window $W$ (or its reversed window $M$, which remains the same). A lighter condition that also ensures Assumptions $1$ and $2$ is a slight strengthening of Assumption $1$:

\vspace{0.2cm}
\begin{flushleft}
\textbf{Assumption 1'.} \textit{For each face $f$ of $M$, the sum $Stab_{\Gamma}(A_f)^*+\mathring{f}$ covers $A_f$, where $\mathring{f}$ denotes the relative interior of $f$.}
\end{flushleft}

\section{\large{\textbf{Preparatory results on the Ellis semigroup of the internal system}}}\label{preliminary}

\subsection{Internal system topology}

\begin{prop}\label{prop:topology.internal.system} The collection of clopen sets $[A]_{\Xi}$, where $ A\in \mathfrak{A}$, forms a subbasis for the topology of the internal system. Moreover, for any pair $A,A'$ in $\mathfrak{A}$ the following Boolean rules are true:
$$[A\cup A']_{\Xi}=[A]_{\Xi}\cup [A']_{\Xi}, \quad [A]_{\Xi}^c=[A^c]_{\Xi}, \quad [A\cap A']_{\Xi}=[A]_{\Xi}\cap [A']_{\Xi}.$$
\end{prop}

\begin{proof} Whenever $w$ is a nonsingular element of $NS\subset \mathbb{R}^n$, one has for any $\gamma \in \Gamma$ that
\begin{align*}\mathfrak{P}(W+w).\gamma :=  \mathfrak{P}(W+w) -\gamma = \mathfrak{P}(W+w+\gamma ^*)
\end{align*}
This is the key that lets us write, for any $\gamma \in \Gamma$, the equalities
\begin{align}\label{cut.up.half.space.translated} [A _f^+]_{\Xi }.\gamma = [A _f^+ +\gamma ^*]_\Xi 
\end{align}
This observation being made, let us start the proof by showing the Boolean equalities. The equality on the left is a simple consequence of closure operation. The equality in the middle is equivalent to having disjoint decompositions
\begin{align}\label{complémentaire} [A _f^+ +\gamma ^*]_\Xi \sqcup [A _f^- +\gamma ^*]_\Xi = \Xi ^\Gamma
\end{align}
which reduces, due to the equalities provided in (\ref{cut.up.half.space.translated}), to showing that $[A _f^+]_\Xi \sqcup [A _f^-]_\Xi = \Xi ^\Gamma $. To that end, note that any element of the internal system $\Xi ^\Gamma$ is the limit of a sequence of nonsingular elements, a sequence which can be taken after extraction into one of the two open half-spaces $A_f^+$ and $A_f ^-$. Therefore such element remains in either $[A _f^+]_\Xi $ or $ [A _f^-]_\Xi$, showing that their union covers the internal system. On the other hand, these subsets are by assumption clopen, so they must have a clopen intersection. Assume for a contradiction that this is not the empty set: It must contain a nonsingular model set $\Lambda $, whose image under $\Pi$ is a nonsingular element $w_\Lambda \in NS\subset \mathbb{R}^n$. However $\Lambda$ is the limit of two sequences of nonsingular model sets, with associated sequences of nonsingular elements in $\mathbb{R}^n$ taken in $A _f^+ $ for the first sequence and in $A _f^-$ for the second one. Taking limits one must have $w_\Lambda \in A_f$, and since $w_\Lambda$ has been taken nonsingular one has the desired contradiction.
\vspace{0.2cm}

Having proven the left and middle Boolean equalities, one can deduce third one by general Boolean algebra.
\vspace{0.2cm}

To show that the sets $[A]_{\Xi}$ where $ A\in \mathfrak{A}$ form a subbasis for the topology, observe that the set $\mathring{M}=-\mathring{W}$ is precisely the intersection of admissible half-spaces $A_f^{s_f}$, where $s_f$ is the sign $-$ or $+$ such that $A_f^{s_f}$ contains the interior of $M$. From this we deduce
\begin{align*} \Xi  = [-\mathring{W}]_\Xi = [\mathring{M}]_\Xi = \left[  \bigcap _{f \text{ face of } W} A_f^{s_f}\right] _\Xi =  \bigcap _{f \text{ face of } W}[A_f^{s_f}]_\Xi .
\end{align*}
Thus the set $\Xi$ and its complementary set can be obtained as finite intersections of sets of the statement. Since $\Xi ^\Gamma $ admits a subbasis formed by the $\Gamma$--translates of $\Xi$ and its complementary set by Proposition \ref{prop:topology.internal.system.tot.disconnected}, the proof is complete.
\end{proof}

\subsection{Cones associated with model sets}

We define the \textit{cut type} of a vector $w\in \mathbb{R}^n$ to be the family of linear hyperplanes for which some parallel singular hyperplane passes through $w$,
\begin{align}\label{type.coupe} \mathfrak{H}_w:= \left\lbrace H_f\in \mathfrak{H}_W \; \vert \; w\in \Gamma ^*+A_f\right\rbrace .
\end{align}

To each $w\in \mathbb{R}^n$ is associated a family of \textit{cones} (also called \textit{corners} in \cite{Le}), which are open cones with vertex $0$ and boundaries formed by hyperplanes in $\mathfrak{H}_w$. We may label each of these cones by a \textit{cone type} $\mathfrak{c} : \mathfrak{H}_W\longrightarrow \left\lbrace -,+,\infty \right\rbrace $, so that the labeled cone is obtained, according to the notations of Section $4$, as
\begin{align*} C:= \bigcap _{H\in \mathfrak{H}_W} H^{\mathfrak{c}(H)}.
\end{align*}
In the above intersection only hyperplanes where $\mathfrak{c}$ has value not equal to $\infty$ are consistent, and we may set the \textit{domain} of a cone type $\mathfrak{c}$ to be the subset $dom(\mathfrak{c})$ of $\mathfrak{H}_W$ where it has value different from $\infty$. Moreover, a cone determined by the cut type $\mathfrak{H}_w$ has only one cone type whose domain is precisely $\mathfrak{H}_w$. Now given a cone $C$ in $\mathbb{R}^n$ and some vector $w$ of $\mathbb{R}^n$, we define
\begin{align}\label{cone.borné} C(w,\varepsilon ):= (C\cap B(0,\varepsilon ))+w
\end{align}
to be the head of the cone $C$, translated at $w$ and of length $\varepsilon $. One can easily verify that a vector $w\in \mathbb{R}^n$ belongs to the nonsingular vectors $NS$ if and only if its cut type $\mathfrak{H}_w$ is empty. In this latter case the unique resulting cone $C$ is the full Euclidean space $\mathbb{R}^n$, for which any set of the form (\ref{cone.borné}) is a Euclidean ball.

\begin{prop}\label{prop:local.topology.internal.system} Given a model set $\Lambda \in \Xi ^\Gamma$, there exists a cone $C_\Lambda$, admitting a cone type $\mathfrak{c}_\Lambda$ with domain $\mathfrak{H}_{w_\Lambda}$, such that the following equivalence holds for each admissible half space $A\in \mathfrak{A}$:
\begin{align*} \Lambda \in [A]_\Xi \quad \Longleftrightarrow  \quad C_\Lambda (w_\Lambda	 ,\varepsilon )\subset A \text{ for some } \varepsilon >0.
\end{align*}
\end{prop}

\begin{proof} Let $\Lambda \in \Xi ^\Gamma$. If $H$ is a hyperplane of the cut type $\mathfrak{H}_{w_\Lambda}$, that is, if one has some $\gamma \in \Gamma $ and some face $f$ with $w_\Lambda \in \gamma ^*+A_f$, $A_f$ parallel to $H$, then the hyperplane $H + w_\Lambda$ is equal to $ \gamma ^*+A_f$ and the half-spaces $H^{\pm}+w_\Lambda$ are admissible. Therefore $[H^{+}+w_\Lambda]_\Xi$ and $[H^{-}+w_\Lambda]_\Xi$ are clopen complementary sets, and the one containing $\Lambda$ defines the sign $\mathfrak{c}_\Lambda(H)$. This provides $\mathfrak{c}_\Lambda$ uniquely. From the Boolean rules stated in Proposition \ref{prop:topology.internal.system}, the model set $\Lambda$ is such that
\begin{align}\label{cone}\Lambda \in \bigcap _{H\in \mathfrak{H}_{w_\Lambda}} \left[ H^{\mathfrak{c}_\Lambda(H)}+w_\Lambda\right] _\Xi= \left[ \bigcap _{H\in \mathfrak{H}_{w_\Lambda}} H^{\mathfrak{c}_\Lambda(H)} + w_\Lambda\right] _\Xi= \left[ C_\Lambda+w_\Lambda\right] _\Xi
\end{align}
with $C_{_\Lambda}$ the unique cone with cone type $\mathfrak{c}_\Lambda$, in particular non-empty. We now show that a model set $\Lambda$ has a neighborhood basis in the internal system obtained as 
\begin{align}\label{nei}
 [C_\Lambda + w_\Lambda]_\Xi \cap \Pi ^{-1}(B(w_\Lambda, \varepsilon))
\end{align}
From the inclusion of $\Lambda$ stated in (\ref{cone}) it is clear that (\ref{nei}) is a family of open neighborhoods of $\Lambda$ in the internal system. We will use the following lemma:

\begin{lem}\label{lem:base.voisinages.ouverts} Let $\pi : X \rightarrow Y$ be a continuous and proper map between locally compact spaces. Let $X_x$ be the fiber of $x$ with respect to $\pi $ for each $x\in X$. If there is a clopen neighborhood $V_x$ of $x$ satisfying $V_x\cap X_x= \lbrace x\rbrace $, then a neighborhood basis of $x$ is provided by $V_x\cap \pi ^{-1}(U)$ with $U$ running among the neighborhoods of $\pi (x)$.
\end{lem}

\begin{proof} Suppose for a contradiction that the stated family is not a neighborhood basis of $x$. One may then select an open neighborhood $V$ of $x$ such that $V_x\cap \pi ^{-1}(U)$ meets $V^c$ for each neighborhood $U$ of $\pi (x)$. Let $\Delta $ be the directed family of open neighborhoods of $\pi (x)$ falling in some compact neighborhood $U_0$ of $\pi (x)$. One may select a net $\lbrace x_U \rbrace_{ U\in \Delta }$ into $V^c$ and with each $x_U $ belonging to $ V_x\cap \pi ^{-1}(U)$. This net falls in the compact set $V_x\cap \pi ^{-1}(U_0)$ and in $V^c$ as well. Taking some accumulation point $x'$, necessarily lying in both $V_x$ and $ X_x$, and in the closed set $V^c$ as well, gives the contradiction as we supposed $V_x\cap X_x= \lbrace x\rbrace $ contained in $V$.
\end{proof}
\vspace{0.2cm}
We then show that a clopen neighborhood of $\Lambda$ which fits the condition of the above lemma is provided by $[C_\Lambda + w_\Lambda]_\Xi$: For this, suppose that $\Lambda$ and $\Lambda '$ are such that $w_\Lambda = w_{\Lambda '}=:w$ in $\mathbb{R}^n$. From Proposition \ref{prop:topology.internal.system} there is a face $f$ of $W$ as well as an element $\gamma \in \Gamma$ such that (up to a permutation of signs $+$ and $-$) $\Lambda\in [A_f^++\gamma ^*]_\Xi$ and $\Lambda '\in [A_f^-+\gamma ^*]_\Xi$. Then the vector $w$ falls into both closed half planes $A_f^{+0}+\gamma ^*$ and $A_f^{-0}+\gamma ^*$, and thus into $A_f + \gamma ^*$. The latter hyperplane can consequently also be written $H_f + w$, and it follows that $\Lambda \in [H_f^+ + w]_\Xi$ whereas $\Lambda '\in [H_f^- + w ]_\Xi$. This shows that $\Lambda '$ is outside $[C_\Lambda + w_\Lambda]_\Xi$, as desired.\\

Now, it is clear that $\Lambda \in [A]_\Xi $ if and only if one has a subset of the form $[C_\Lambda + w_\Lambda]_\Xi \cap \Pi ^{-1}(B(w_\Lambda, \varepsilon))$ included in $[A]_\Xi $ for some $\varepsilon >0$. Then intersecting with $NS$ gives that $C_\Lambda (w_\Lambda ,\varepsilon )\cap NS$ falls into $ A\cap NS$, and by taking closure and next interior in $\mathbb{R}^n$ one obtains the right-hand inclusion of the statement. Conversely if the right-hand inclusion of the statement occurs for some $\Lambda \in \Xi ^\Gamma$ then we may choose a sequence of nonsingular model sets converging to it, in a manner that the associated sequence of nonsingular vectors falls into (\ref{nei}), and thus into $C_\Lambda (w_\Lambda ,\varepsilon )$. The sequence of nonsingular model sets then lies into $[A]_\Xi $, and since this latter is closed we obtain the result.
\end{proof}

\subsection{Topology of the internal system Ellis semigroup}

Recall that by construction, the Ellis semigroup for the internal system is a closure of the group $\Gamma$, or rather the resulting group of homeomorphisms on the internal system. Thus for any Euclidean subset $A$ one may set a corresponding subset $[A]_E$ to be the closure of $\left\lbrace \gamma \in \Gamma \, \vert \, \gamma ^*\in A\right\rbrace $ in the Ellis semigroup $E(\Xi ^\Gamma,\Gamma )$ of the internal system. We would in fact consider a specific family of Euclidean subsets, namely
\begin{align*}\mathfrak{A}_{\mathrm{Ellis}}:=\lbrace H^\mathfrak{t} +w\, \vert \, H\in \mathfrak{H}_W, \, \mathfrak{t}\in \lbrace -,0,+\rbrace, \, w\in \mathbb{R}^n\rbrace .
\end{align*}
Observe that the above family contains the family $\mathfrak{A}$ of admissible half-spaces, in a strict sense however.

\begin{prop}\label{prop:topology.ellis.internal.system} Any set $[A]_E$, where $ A\in \mathfrak{A}_{Ellis}$, is clopen, and the collection of these sets forms a subbasis for the topology of $E(\Xi ^\Gamma,\Gamma )$. Moreover, for any pair $A,A'$ in $\mathfrak{A}_{\mathrm{Ellis}}$ the following Boolean rules are true:
$$[A\cup A']_E=[A]_E\cup [A']_E, \quad [A]_E^c=[A^c]_E, \quad [A\cap A']_E=[A]_E\cap [A']_E .$$
\end{prop}

\begin{proof} From Proposition \ref{prop:topology.internal.system}, the sets $[A]_\Xi$, where $A$ is an admissible half-space, are clopen subsets of the internal system $\Xi ^\Gamma$, and form a subbasis for its topology. It thus follows that the sets
\begin{align*} V\left( \Lambda , [A]_\Xi\right) := \left\lbrace g \in E(\Xi ^\Gamma,\Gamma ) \, \vert \; \Lambda .g\in [A]_\Xi \right\rbrace ,
\end{align*}
where $\Lambda $ is any model set in the internal system and $A$ is any admissible half-space, are clopen subsets of the Ellis semigroup $E(\Xi ^\Gamma,\Gamma )$, and that they form a subbasis for its topology. Moreover, using the fact that $[\gamma ^*+A_f^{\pm}]_\Xi$ is equal to $[A_f^{\pm}]_\Xi .\gamma$ whatever the element $\gamma \in \Gamma$ one can directly check that $V(\Lambda , [\gamma ^*+A_f^{\pm}]_\Xi)$ is equal to $V(\Lambda .(-\gamma ), [A_f^{\pm}]_\Xi)$. This shows that a subbasis for the Ellis semigroup topology is obtained as the collection
\begin{align}\label{sub-basis} \left\lbrace V\left( \Lambda , [A_f^{\pm}]_\Xi\right) \, \vert \; \Lambda \in \Xi ^\Gamma , \, f \text{ face of } M\right\rbrace .
\end{align}
In order to relate these sets with the ones given in the statement we prove a cornerstone lemma to this proposition:

\begin{lem}\label{lem:cornerstone} Let $\Lambda$ be in the internal system $\Xi ^\Gamma$. Then
 \begin{align*} V\left( \Lambda , [A_f^+]_\Xi\right)= \begin{cases}  [A^{+0}_f-w_\Lambda]_E & \text{ if } \mathfrak{c}_\Lambda(H_f)=+, \\ [A^{+}_f-w_\Lambda]_E & \text{ if } \mathfrak{c}_\Lambda(H_f)=-, \\  [A^{+0}_f-w_\Lambda]_E= [A^{+}_f-w_\Lambda]_E & \text{ if } \mathfrak{c}_\Lambda(H_f)=\infty . 
\end{cases}
 \end{align*}
The same statement holds with the $+$ and $-$ signs switched everywhere.
\end{lem}

\begin{proof} Recall from Lemma \ref{lem:clopen} that a clopen set of $E(\Xi ^\Gamma,\Gamma )$ is the closure of its subset of $\Gamma$--elements. Now given $V( \Lambda , [A_f^+]_\Xi )$, an element $\gamma \in \Gamma$ lies inside if and only if $\Lambda .\gamma\in [A_f^+]_\Xi$, which happens from Proposition \ref{prop:local.topology.internal.system} if and only if $C_{\Lambda .\gamma}(w_{\Lambda .\gamma}, \varepsilon ) $ embeds into $A^+_f$ for some $\varepsilon >0$. As the cones of $\Lambda$ and its $\gamma$--translate are the same, and because the factor map $\Pi$ is $\Gamma$--equivariant, the previous condition is equivalent to
\begin{align}\label{C3} C_\Lambda(\gamma ^*,\varepsilon ) \subset A^+_f -w_\Lambda
\end{align}
for some $\varepsilon >0$. It is then obvious that:

\begin{itemize}
\item[•] Whenever $\gamma ^*\in A^+_f -w_\Lambda$ this condition is satisfied.

\item[•] Whenever $\gamma ^*\in A^-_f -w_\Lambda$ this condition is not satisfied.

\end{itemize}

Now suppose that $\mathfrak{c}_\Lambda(H_f)=\infty$, so that $H_f$ doesn't belong to the cut type of $w_\Lambda$: Then no element of $\Gamma$ maps to $A_f -w_\Lambda$ under the $*$--map, and thus by taking closure in the Ellis semigroup one has the desired equality in the case $\mathfrak{c}_\Lambda(H_f)=\infty$.
\vspace{0.2cm}

Suppose by contrast that $\mathfrak{c}_\Lambda(H_f)\neq \infty$, so that there exist elements of $\Gamma$ whose image under the $*$--map falls into $A_f -w_\Lambda$. Then for each such $\gamma \in \Gamma$ the hyperplane $A_f-w_\Lambda$ may also be written $H_f+\gamma ^*$, giving $A^+_f-w_\Lambda=H^+_f+\gamma ^*$. Hence such a $\gamma$ satisfies (\ref{C3}) if and only if the cone $C_\Lambda$ lies into $H_f^+$, which can be rewritten as $\mathfrak{c}_\Lambda(H_f)=+$. Again by taking closure in the Ellis semigroup, one has the desired equalities in the case $\mathfrak{c}_\Lambda(H_f)\neq \infty$.
\vspace{0.2cm}

The argument remains valid when interchanging the $+$ and $-$ signs everywhere, completing the proof of the lemma.
\end{proof} 

\begin{lem}\label{lem:ellis.clopen} For each hyperplane $H$ and vector $w\in \mathbb{R}^n$, one has a partition of the Ellis semigroup by clopen sets
\begin{align}\label{partition}
E(\Xi ^\Gamma ,\Gamma )=\left[ H^-  +w\right] _E\sqcup \left[ H^{} +w\right] _E\sqcup \left[ H^+  +w\right] _E .
\end{align}
\end{lem}

\begin{proof} First observe that by construction the group $\Gamma$ is dense in the Ellis semigroup, and consequently the union of the three right-hand sets stated in the equality must covers the Ellis semigroup. Now select a face $f$ with $H=H_f$ and let $w'\in \mathbb{R}^n$ be such that $H^\mathfrak{t}+w$ can be rewritten as $A_f^\mathfrak{t}-w'$ for each sign $\mathfrak{t}\in \lbrace -,0,+\rbrace$ (this can always be achieved as $H$ and $A_f$ are parallel). This choice of vector $w'$ will be kept along this proof. It is quite clear that the middle term $[H_f  +w]_E$ is nonempty if and only if one has an element $\gamma \in \Gamma$ such that $\gamma ^*\in H_f  +w$, or equivalently in $A_f-w'$, which in turns exactly means that $H$ is a hyperplane of the cut type $\mathfrak{H}_{w'}$. Thus we will consider two cases:
\vspace{0.2cm}

Suppose that $H\in \mathfrak{H}_{w'}$: we may select two cones, both determined by the cut type $\mathfrak{H}_{w'}$, living at opposite sides with respect to $H$. Let us pick two model sets $\Lambda$ and $\Lambda '$ with common associated vector $w'$ in $\mathbb{R}^n$ and associated with these cones, so that $\mathfrak{c}_\Lambda(H)=+$ and $\mathfrak{c}_{\Lambda '}(H)=-$ up to a switch of signs (the existence of such model sets is shown in Theorem \ref{theo:topologie.par.cones} appearing later, whose proof is independent of the present statement). Then by the previous lemma, the set $[H_f^-  +w]_E$ is the clopen subset $V( \Lambda , [A_f^-]_\Xi )$, and is disjoint from the other two since they are both included in $V( \Lambda , [A_f^+]_\Xi )$. In the same way the set $[H_f^+  +w]_E$ is the clopen subset $V( \Lambda ', [A_f^+]_\Xi )$, and is disjoint from the two others since they are both included into $V( \Lambda ', [A_f^-]_\Xi )$. As the left-hand term and the right-hand term are clopen and disjoint from the respective two other sets then the stated union must be disjoint, and the middle term is clopen as well.
\vspace{0.2cm}

If $H\notin \mathfrak{H}_{w'}$ then things are even easier: the middle term becomes empty, and in pretty much the same way as before, by picking only one model set with associated vector $w'$ one can show that the two sets of the union are clopen and disjoint.
\end{proof}

Now the proof of the statement almost immediately follows: By Lemma \ref{lem:ellis.clopen} the sets of the statement are clopen sets, and form a subbasis since any subset of the family (\ref{sub-basis}) can be written as one of them by Lemma \ref{lem:cornerstone}. It remains to show the Boolean rules: The left-hand rule is a direct consequence of the closure operation, whereas the middle rule follows from the family of partitions given by Lemma \ref{lem:ellis.clopen}. The third rule naturally follows from the two others.
\end{proof}

\section{\large{\textbf{Main result on the internal system Ellis semigroup}}}\label{main internal}

\subsection{The face semigroup of a convex polytope}

Given a real cut and project scheme $(\mathbb{R}^n, \Sigma, \mathbb{R}^d)$ with an almost canonical window $W$ in the internal space, we shall define the \textit{face semigroup} of $W$ \cite{Bro, Sa}.

Let $\mathfrak{H}_W$ be the family of linear hyperplanes parallel to the faces of $W$. It defines a stratification of $\mathbb{R}^n$ by cones of dimension between $0$ and $n$ (those cones are called \textit{faces} in \cite{Sa}), that is, by nonempty sets of the form
\begin{align}\label{face} \bigcap _{H\in \mathfrak{H}_W} H^{\mathfrak{t}(H)} .
\end{align}
where $\mathfrak{t}(H)$ is a symbol among $ \left\lbrace -,0,+\right\rbrace $ for each $ H\in \mathfrak{H}_W$. Then each such cone $C$ is determined through a unique map $\mathfrak{t}_C: \mathfrak{H}_W\longrightarrow \left\lbrace -,0,+\right\rbrace $, which we call here its \textit{cone type}. A special class of cones is that of \textit{chambers}, that is, the cones of maximal dimension $n$, which are open in $\mathbb{R}^n$ and are precisely those with a nowhere-vanishing cone type. On the other extreme is the unique cone of dimension $0$, namely the singleton $\left\lbrace  0\right\rbrace $, whose type is entirely vanishing and which we denote by $\mathfrak{o}$.
\vspace{0.2cm}

Let us denote by $\mathfrak{T}_W$ the above set of cones, and define on this set a semigroup law: if $C,C'\in \mathfrak{T}_W$ are given, then the product $C.C'$ is the face whose type is given by
\begin{align*} \mathfrak{t}_{C.C'}(H)=\mathfrak{t}_C.\mathfrak{t}_{C'}(H):= \begin{cases} \mathfrak{t}_{C'}(H) \text{ if }\mathfrak{t}_C(H)= 0,  \\ \mathfrak{t}_C(H)\, \text{  else.} \\ \end{cases}
\end{align*}

The reading direction is from right to left, as for actions: first we look at the value of $\mathfrak{t}_{C'}(H)$, we keep it when $\mathfrak{t}_C(H)= 0$ and else replace it by $\mathfrak{t}_C(H)$, which in this case makes us forget the existence of $\mathfrak{t}_{C'}$. It may easily checked that this product law is well defined on $\mathfrak{T}_W$, that is, the product of two (non empty) cones is again a (non empty) cone, and is associative.

\begin{de}\label{def:face.semigroup} The face semigroup associated with the polytope $W$ in $\mathbb{R}^n$ is the set $\mathfrak{T}_W$ equipped with the above product law.
\end{de}

It is clear from the formula that $\mathfrak{o}$ is an identity for $\mathfrak{T}_W$. Moreover, any cone $C$ satisfies the equality $C.C=C$, that is, is \textit{idempotent} in $\mathfrak{T}_W$. There moreover exists a natural partial order on the face semigroup under which $C\leqslant C'$ if and only if $C'$ is a lower-dimensional facet of $C$, or equivalently when the inclusion $C'\subseteq \overline{C}$ occurs. This may be rephrased by means of the semigroup law on $\mathfrak{T}_W$, as we have
\begin{align*} C\leqslant C' \quad \Longleftrightarrow \quad \mathfrak{t}_C = \mathfrak{t}_{C'}.\mathfrak{t}_C
\end{align*}

With respect to this order, the chambers are the minimal cones whereas $\mathfrak{o}$ is the (unique) maximal cone in the face semigroup. Some authors use the reverse order instead, but it appears more convenient for later needs to define the order as above.

\subsection{Taking $\Gamma$ into account}

Here we introduce a modified version of the face semigroup obtained from an almost canonical window $W$ of the internal space $\mathbb{R}^n$ of some real cut and project scheme.
\vspace{0.2cm}

Let us call a cone $C$ of the face semigroup \textit{nontrivial} whenever the origin in $\mathbb{R}^n$ is an accumulation point of elements of $C\cap \Gamma ^*$. We moreover denote the family of nontrivial cones of the face semigroup by $\mathfrak{T}_{W,\Gamma }$, and refer it as the \textit{nontrivial face semigroup}. It is at this point not clear whether $\mathfrak{T}_{W,\Gamma }$ is a subsemigroup of $\mathfrak{T}_{W}$. However to convince ourselves that it is the case, we may observe that the product $C.C'$ of two cones of the face semigroup is the only cone containing a small head of the cone $C'$ when this latter is shifted by a small vector of $C$, and that this preserves the subset $\mathfrak{T}_{W,\Gamma }$ in the face semigroup.
\vspace{0.2cm}

Now given a nontrivial cone $C$, as $C\cap \Gamma ^*$ accumulates at $0$, the vector space $\langle C\rangle $ spanned by $C$ admits a subgroup $\langle C\rangle \cap \Gamma ^*$ which cannot be uniformly discrete, and thus is "dense along some subspace". More precisely we state in our setting a theorem of \cite{Se}:

\begin{theo}\label{theo:Senechal} The vector space $\langle C\rangle $ uniquely decomposes as a direct sum $ V\oplus D$, where $V\cap \Gamma ^*$ is dense in $ V$, $D\cap \Gamma ^*$ is uniformly discrete in $D$, and $\langle C\rangle \cap \Gamma ^* = \left( V \cap \Gamma ^*\right) \oplus \left(  D \cap \Gamma ^*\right) $.
\end{theo}

Now given a nontrivial cone $C$ of the face semigroup with decomposition $\langle C\rangle = V\oplus D$ provided by the previous theorem, the summand $V$ is nontrivial and thus one may attach to it another smaller cone,
\begin{align*} \mathsf{C}:= C\cap V.
\end{align*}
We call $\mathsf{C}$ the \textit{plain cone} associated to $C$. By construction, the cone $\mathsf{C}$ is open in the space $V$ and spans this latter space, and $\mathsf{C}\cap \Gamma ^*$ is a dense subset of the plain cone $\mathsf{C}$. It is easy to observe that $\mathsf{C}=C$ when and only when the set $C\cap \Gamma ^*$ is dense in the cone $C$. For any nontrivial cone type $\mathfrak{t}\in \mathfrak{T}_{W, \Gamma }$ we may define $\mathsf{C}_\mathfrak{t}$ to be the plain cone associated with $C_\mathfrak{t}$.

\subsection{The main theorem for internal system Ellis semigroup}

Let us consider an Ellis transformation $g\in E(\Xi ^\Gamma,\Gamma)$ with associated translation vector $w_g$ in $\mathbb{R}^n$. Given a hyperplane $H\in \mathfrak{H}_W$, it has been shown in Lemma \ref{lem:ellis.clopen} that the mapping $g$ falls into one and only one clopen subset of the form $[H^{\mathfrak{t}} +w_g]_E$, whose sign for any hyperplane $H\in \mathfrak{H}_W$ determines a face type $\mathfrak{t}_g$ uniquely. To see that $\mathfrak{t}_g$ is a face type in the above sense, that is, is associated with a nonempty cone $C_g$ of the stratification obtained from $\mathfrak{H}_W$, observe that from the Boolean rules of Proposition \ref{prop:topology.ellis.internal.system} one has
\begin{align}\label{intersection.ellis} g\in  \bigcap _{H \in \mathfrak{H}_W} \left[ H^{\mathfrak{t}_g(H)} +w_g\right] _E = \left[ \bigcap _{H \in \mathfrak{H}_W} H^{\mathfrak{t}_g(H)} +w_g\right] _E = \left[ C_g +w_g\right] _E,
\end{align}
which ensure that $C_g$ must be nonempty. Having related the internal system Ellis semigroup with the face semigroup just defined, we are now able to set our main theorem concerning the internal system Ellis semigroup:

\begin{theo}\label{theo:principal.interne} The mapping associating to any transformation $g$ the couple $(w_g,\mathfrak{t}_g)$ establishes an isomorphism between the Ellis semigroup $E(\Xi ^\Gamma,\Gamma)$ and the subsemigroup of the direct product $\mathbb{R}^n \times \mathfrak{T}_{W,\Gamma  }$ given by
\begin{align*}\bigsqcup _{\mathfrak{t}\in \mathfrak{T}_{W,\Gamma }} \left[ \langle \mathsf{C}_\mathfrak{t}\rangle +\Gamma ^*\right] \times \left\lbrace \mathfrak{t}\right\rbrace .
\end{align*}
This isomorphism becomes a homeomorphism when the above union is equipped with the following convergence class: $(w_\lambda,\mathfrak{t}_\lambda)\longrightarrow (w,\mathfrak{t})$ if and only if
\begin{align*} \forall \varepsilon >0, \exists \, \delta _\lambda >0 \text{ such that  } \; \mathsf{C}_{\mathfrak{t}_\lambda}(w_\lambda,\delta _\lambda)\subset \mathsf{C}_{\mathfrak{t}}(w,\varepsilon ) \; \text{  for large enough }\lambda .
\end{align*}
The Ellis semigroup $E(\Xi ^\Gamma,\Gamma)$ has a first countable topology.
\end{theo}

The convergence class of the statement is there to precise the full family of nets and limit points which obey the above condition. This family completely characterizes the Ellis semigroup topology since, being derived from the topology of the internal system Ellis semigroup, it satisfies a correct set of axioms which permit to recover the closure operator on the Ellis semigroup, and thus its topology (see \cite{Kelley}).
\vspace{0.2cm}

The remaining part of this section is devoted to the proof of the above theorem. To this end we decompose the proof into three parts: the first one states the existence of a semigroup isomorphism between the internal system Ellis semigroup and a subsemigroup of the direct product $\mathbb{R}^n \times \mathfrak{T}_{W}$. The second step states the proof that the isomorphic image maps into $\mathbb{R}^n \times \mathfrak{T}_{W,\Gamma  }$ and is of the form stated above. In a third part we then show the topological part of the statement.

\paragraph{6.3.1 $\; $ Step 1: Existence of the semigroup isomorphism}

\begin{prop}\label{prop:morphism.type} The mapping $E(\Xi ^\Gamma,\Gamma) \longrightarrow \mathfrak{T}_W$ that associates to each transformation $g$ its face type $\mathfrak{t}_g$ is a semigroup morphism.
\end{prop}

\begin{proof} We have to show that given two transformations $g$ and $h$ the face types $\mathfrak{t}_{g.h}$ and $\mathfrak{t}_g.\mathfrak{t}_h$ are equal. By (\ref{intersection.ellis}) the transformation $g.h$ lies in the clopen subset $[C_{g.h}+w_{g.h}]_E$. Since, by construction, $\Gamma$ is dense in the Ellis semigroup, and since the composition law on this latter is right-continuous, one can find a $\gamma \in \Gamma$ sufficiently close to $h$ in the sense that
 $$  \text{(i)}\; \; \gamma \in [C_h + w_h ]_E,  \; \; \; \; \; \; \; \; \text{(ii)}\; \; g.\gamma \in [C_{g.h}+w_{g.h}]_E . $$
From Lemma \ref{lem:ellis.clopen} together with the Boolean rules of Proposition \ref{prop:topology.ellis.internal.system}, one can deduce from (i) that $\gamma ^*\in C_h + w_h$, or equivalently $(\gamma ^*-w_h)\in C_h$ in the internal space. Moreover, as the transformation $g.\gamma$ lies both in the clopen subset $[C_{g.\gamma}+w_{g.\gamma}]_E$ and the open subset $(\Pi ^*)^{-1}(B(w_{g.\gamma },\varepsilon))$, again from the density of $\Gamma $ in the Ellis semigroup together with point (ii), one can find an element $\gamma _\varepsilon \in \Gamma$ sufficiently close to $g.\gamma$ so that
\begin{align*} \gamma ^*_\varepsilon \in \left( C_{g.h}+w_{g.h}\right) \cap \left( C_{g.\gamma}+w_{g.\gamma} \right) \cap B(w_{g.\gamma },\varepsilon).
\end{align*}
Since the cone associated with $g.\gamma$ is equal to the one associated with $g$, the previous fact implies that
\begin{align*} C_g(\gamma ^*-w_h,\varepsilon ) \cap C_{g.h}\neq \emptyset \; \forall \varepsilon >0 \; \;  \text{with} \; \;  \gamma ^*-w_h\in C_h .
\end{align*} 
Let us now consider three cases about a hyperplane $H\in \mathfrak{H}_W$:
\vspace{0.2cm}

$ \mathfrak{t}_h(H) =+  \;$  In this case the vector $\gamma ^*-w_h\in C_h $ falls into the open half-space $ H^+$, and thus one may find a $\varepsilon _0$ with $C_g(\gamma ^*-w_h,\varepsilon _0)$ included in $ H^+$, so that $H^+$ must intersect the cone $C_{g.h}$. This forces $C_{g.h}\subset H^+$, or equivalently $\mathfrak{t}_{g.h}(H)=+$.
\vspace{0.2cm}

$\mathfrak{t}_h(H)=-\; $  By the same type of argument one can show that $\mathfrak{t}_{g.h}(H)=-$.
\vspace{0.2cm}

$\mathfrak{t}_h(H)=0 \;$  In this last case one has $\gamma ^*-w_h\in H$ and thus $C_g(\gamma ^*-w_h,\varepsilon )\subset H^{\mathfrak{t}_g(H)}$ whatever the symbol $\mathfrak{t}_g(H)$. It thus follows that $H^{\mathfrak{t}_g(H)}\cap C_{g.h}$ is nonempty, which necessary gives $C_{g.h}\subset H^{\mathfrak{t}_g(H)}$, or equivalently $\mathfrak{t}_{g.h}(H)=\mathfrak{t}_g(H)$.
\vspace{0.2cm}

The above three cases show that the cone type $\mathfrak{t}_{g.h}$ is equal to the composition $\mathfrak{t}_g.\mathfrak{t}_h$, as desired.
\end{proof}

Combining the previous proposition with the existence of the onto morphism of Proposition \ref{prop:internal.morphism.ellis}, we see that the mapping that associates to each transformation $g$ in $E(\Xi ^\Gamma,\Gamma)$ the couple $(w_g, \mathfrak{t}_g)$ in the product semigroup $\mathbb{R}^n\times \mathfrak{T}_W$ is a semigroup morphism. Thus to settle Step $1$, we only need to show injectivity:
\vspace{0.2cm}

Suppose for that purpose that two transformations $g$ and $h$ satisfy $w_g=w_h=:w$ in the internal space. Then by using the subbasis of Proposition \ref{prop:topology.ellis.internal.system} one can find a vector $w_0$ as well as a hyperplane $H\in \mathfrak{H}_W$ such that $g$ and $h$ fall into different clopen subsets among the partition
\begin{align*}
E(\Xi ^\Gamma ,\Gamma )=\left[ H^-  +w_0\right] _E\sqcup \left[ H^{} +w_0\right] _E\sqcup \left[ H^+  +w_0\right] _E .
\end{align*}
Thus one must have that $w$ and $w_0$ are equal up to a vector of the hyperplane $H$, and this implies that the signs $\mathfrak{t}_g(H)$ and $\mathfrak{t}_h(H)$ must be different. This exactly means that the associated cone types $\mathfrak{t}_g$ and $\mathfrak{t}_h$ are different, and the proof of Step $1$ is complete.

\paragraph{6.3.2 $\; $  Step 2: Determination of the isomorphic image} The problem now is to identify the subsemigroup of $\mathbb{R}^n \times \mathfrak{T}$ isomorphic to the internal system Ellis semigroup via the previous mapping. To that end, one may describe this subsemigroup as a disjoint union
\begin{align*} 
\bigsqcup _{\mathfrak{t}\in \mathfrak{T}_W} \mathbb{R}^n_{\mathfrak{t}} \times \left\lbrace \mathfrak{t}\right\rbrace 
\end{align*}
for some Euclidean subsets $\mathbb{R}^n_\mathfrak{t}$, the \textit{allowed translations} of a cone type $\mathfrak{t}$, which we need to identify. A first point about this is the following lemma.

\begin{lem}\label{lem:translations.permises} For any cone type $\mathfrak{t}\in \mathfrak{T}_W$ with associated cone $C_\mathfrak{t}$ one has
\begin{align*} \mathbb{R}^n_\mathfrak{t} = \left\lbrace w\in \mathbb{R}^n \, \vert \, (C_{\mathfrak{t}}+w)\cap \Gamma ^* \text{ accumulates at } w\right\rbrace .
\end{align*}
\end{lem}

\begin{proof} Given some $\mathfrak{t}\in \mathfrak{T}_W$ with associated cone $C_\mathfrak{t}$, its set of allowed translations $\mathbb{R}^n_\mathfrak{t}$ is by construction $ \mathbb{R}^n_\mathfrak{t} = \left\lbrace  w_g \, \vert \, g\in E(\Xi ^\Gamma ,\Gamma) \text{ and }  \mathfrak{t}_g=\mathfrak{t} \right\rbrace $.
\vspace{0.2cm}

Let us show $"\supseteq "$: If $w$ is such that $(C_{\mathfrak{t}}+w)\cap \Gamma ^*$ accumulates at $ w$ then the intersection $(C_{\mathfrak{t}}+w)\cap B(w, \varepsilon )\cap \Gamma ^* $ is non-empty for any $\varepsilon > 0$, and thus the family $\left\lbrace [C_{\mathfrak{t}}+w]_E\cap (\Pi ^*)^{-1}(B(w,\varepsilon ))\right\rbrace _{\varepsilon > 0}$ forms a filter base in the space $E(\Xi ^\Gamma,\Gamma)$. In turn, the morphism $\Pi ^*$ is, by Proposition \ref{prop:map.ellis.loc.compact}, a proper map so this filter base, for $0< \varepsilon < \varepsilon _0$, lies in the fixed compact subset $(\Pi ^*)^{-1}\left( \overline{B}(w,\varepsilon _0)\right) $ and thus possess an accumulation point $g$. This Ellis transformation necessarily satisfies $w_g=w$, and because the set $[C_{\mathfrak{t}}+w]_E= [C_{\mathfrak{t}}+w_g]_E$ is closed, containing the above filter base, it thus contains $g$. We deduce that $C_g= C_{\mathfrak{t}}$, or equivalently $\mathfrak{t}_g = \mathfrak{t}$, giving that $w = w_g \in \mathbb{R}^n_\mathfrak{t}$.
\vspace{0.2cm}

Conversely we show $"\subseteq "$: Given some cone type $\mathfrak{t}$ and some Ellis transformation $g$ with $\mathfrak{t}= \mathfrak{t}_g$, then as $g$ lies in $[C_g +w_g]_E$ one can select a net of elements of $(C_g +w_g)\cap \Gamma ^*=   (C_\mathfrak{t} +w_g)\cap \Gamma ^*$ converging to $g$ in the internal system Ellis semigroup. Applying $\Pi ^*$ we obtain a net of $(C_\mathfrak{t} +w_g)\cap \Gamma ^* $ converging to $w_g$ in the Euclidean space $\mathbb{R}^n$, so that $ (C_{\mathfrak{t}}+w_g)\cap \Gamma ^*$ accumulates at $w_g$.
\end{proof}

Let now $\mathfrak{T}_{W,0}$ be the homomorphic image of the internal system Ellis semigroup in the face semigroup $\mathfrak{T}_{W}$ via the morphism of Proposition \ref{prop:morphism.type}. Then it precisely consists of those cone types that have a nonempty associated subset $\mathbb{R}^n_\mathfrak{t}$ of allowed translations. From the definition of the plain face semigroup $\mathfrak{T}_{W,\Gamma }$, a face type $\mathfrak{t}$ is nontrivial if and only if $0$ lies in $\mathbb{R}^n_\mathfrak{t}$, which shows in particular that $\mathfrak{T}_{W,0}$ contains the plain face semigroup $\mathfrak{T}_{W,\Gamma }$.
\vspace{0.2cm}

We will now write any Euclidean subset $\mathbb{R}^n_\mathfrak{t}$ in a more suitable form. Obviously it is sufficient to consider cone types of the homomorphic image $\mathfrak{T}_{W,0}$. Observe that for any such cone type, thei associated Euclidean subset of allowed translations is stable under $\Gamma ^*$--translation.

\begin{prop}\label{prop:décomp.translations.permises} Let $\mathfrak{t}\in \mathfrak{T}_{W,0}$, with $\langle C_\mathfrak{t} \rangle = V_\mathfrak{t} \oplus D_\mathfrak{t}$ being its direct sum decomposition from Theorem \ref{theo:Senechal}. Then one has
\begin{align*} \mathbb{R}^n_{\mathfrak{t}}= V_{\mathfrak{t}}+ \Gamma ^*.
\end{align*}
\end{prop}

\begin{proof} For $\mathfrak{t}\in \mathfrak{T}_{W,0}$ and $\langle C_\mathfrak{t} \rangle = V_\mathfrak{t} \oplus D_\mathfrak{t}$, denote by $P^V$ (resp. $P^D$) the skew projection of $\langle C_\mathfrak{t} \rangle$ with range $V_{\mathfrak{t}}$ and kernel $D_{\mathfrak{t}}$ (resp. the skew projection with range $ D_{\mathfrak{t}}$ and kernel $V_{\mathfrak{t}}$). Then, from the particular form of the decomposition, one has $P^V(\langle C_\mathfrak{t} \rangle\cap \Gamma ^*)=  V_{\mathfrak{t}}\cap \Gamma ^*$ and $P^D( \langle C_\mathfrak{t} \rangle\cap \Gamma ^*)= D_{\mathfrak{t}}\cap \Gamma ^*$.
\vspace{0.2cm}

Let us show first that $\mathbb{R}^n_{\mathfrak{t}}$ lies in $ \langle C_\mathfrak{t} \rangle + \Gamma ^*$: Any vector $w\in \mathbb{R}^n_{\mathfrak{t}}$ admits some $\gamma ^*$ in $(C_{\mathfrak{t}}+w)\cap \Gamma ^*$, so that $\gamma ^*-w$ lies in $C_{\mathfrak{t}}$ and thus in $\langle C_\mathfrak{t} \rangle$. So does the vector $w-\gamma ^*$, giving that $w$ lies in $ \langle C_\mathfrak{t} \rangle + \Gamma ^*$.
\vspace{0.2cm}

Now we more precisely show that $\mathbb{R}^n_{\mathfrak{t}}$ lies in $ V_{\mathfrak{t}}+ \Gamma $:
Given $w\in \mathbb{R}^n_{\mathfrak{t}}$, one may write $w= w' + \gamma ^*$ with $w'\in \langle C_\mathfrak{t} \rangle$ and $\gamma \in \Gamma ^*$, $w'$ itself being in $\mathbb{R}^n_{\mathfrak{t}}$ as this latter is stable under $\Gamma ^*$--translation. It thus suffices to prove that $w'$ lies in $V_{\mathfrak{t}}+ \Gamma $ to conclude. From the previous lemma, $w'$ is the limit point of a sequence $(\gamma ^*_k)$ of elements in $(C_{\mathfrak{t}}+w')\cap \Gamma ^*$, in turn included in $\langle C_\mathfrak{t} \rangle \cap \Gamma ^*$. Thus $P^D(\gamma ^*_k)$ converges to $P^D(w')$ and $P^V(\gamma ^*_k)$ converges to $P^V(w')$. But as the sequence $(P^D(\gamma ^*_k))$ lies in the uniformly discrete subset $ D_{\mathfrak{t}}\cap \Gamma ^*$ of $ D_{\mathfrak{t}}$, it must be eventually constant, equal to $P^D(w')$ for great enough $k$. Hence $P^D(w')$ lies in $\Gamma ^*$, which gives $w'= P^W(w')+ P^D(w')\in V_{\mathfrak{t}}+ \Gamma ^*$, as desired.
\vspace{0.2cm}

We want to observe that the sequence eventually satisfies $P^V(\gamma ^*_k)= \gamma ^*_k- P^D(w')\in (C_{\mathfrak{t}}+w')\cap \Gamma ^*- P^D(w')$, with $P^D(w')\in \Gamma ^*$, and thus $P^V(\gamma ^*_k)\in (C_{\mathfrak{t}}+P^V(w'))\cap \Gamma ^*$. Hence $P^V(\gamma ^*_k)-P^V(w')= P^V(\gamma ^*_k-w')$ lies in both $ V_{\mathfrak{t}}$ and $C_{\mathfrak{t}}$ eventually, which ensures that the intersection $\mathsf{C}_\mathfrak{t}:=C_{\mathfrak{t}}\cap V_{\mathfrak{t}}$ is nonempty.
\vspace{0.2cm}

Now we show that $\mathbb{R}^n_{\mathfrak{t}}$ contains $ V_{\mathfrak{t}}+ \Gamma$: To that end it suffices from $\Gamma ^*$--invariance to show that it contains $ V_{\mathfrak{t}}$. First it is clear that the subset $\mathsf{C}_\mathfrak{t}$ is a (nonempty) open cone of the space $V_{\mathfrak{t}}$, since is the intersection of  $C_{\mathfrak{t}}$ which is open in its own spanned space $\langle C_\mathfrak{t} \rangle $ with the subspace $V_{\mathfrak{t}}$. Let now $ w\in V_{\mathfrak{t}}$ be given. Then $\mathsf{C}_\mathfrak{t}$ is open in $W_{\mathfrak{t}}$ and is a cone pointed at $0$, so that $\mathsf{C}_\mathfrak{t}+w$ is an open cone of $V_{\mathfrak{t}}$ pointed at $w$. But from the density of $V_\mathfrak{t}\cap \Gamma ^*$ in $V_{\mathfrak{t}}$ one can obtain $w$ as an accumulation point of $\mathsf{C}_\mathfrak{t}+w \cap \Gamma ^*$ and thus of $(C_{\mathfrak{t}}+w)\cap \Gamma ^*$, showing that $w\in \mathbb{R}^n_{\mathfrak{t}}$, as desired.
\end{proof}

From the previous proposition one gets that any cone type $\mathfrak{t}$ of $\mathfrak{T}_{W,0}$ has the origin $0$ as allowed translation, and thus is an element of $\mathfrak{T}_{W,\Gamma}$. This shows that the internal system Ellis semigroup is isomorphic with a subsemigroup of the direct product $\mathbb{R}^n\times \mathfrak{T}_{W,\Gamma}$, and that its isomorphic image is of the form stated in Theorem \ref{theo:principal.interne}, once we recall that $V_\mathfrak{t}$ is spanned by the pain cone $\mathsf{C}_\mathfrak{t}$ for any $\mathfrak{t}\in \mathfrak{T}_{W,\Gamma}$. This completes Step $2$.

\paragraph{6.3.3 $\; $ Step 3: The topology of convergence}

Let us first show the first countability property of the internal system Ellis semigroup: From the injectivity of the mapping associating to any transformation $g$ the couple $(w_g, \mathfrak{t}_g)$, one can deduce that $g$ is the only transformation in its fiber with respect to $\Pi ^*$ falling into the clopen subset $\left[ C_g +w_g\right] _E$ of the Ellis semigroup. It follows by Lemma \ref{lem:base.voisinages.ouverts} that a neighborhood basis of $g$ is provided by the intersections 
\begin{align}\label{ellis.local.basis}\left[ C_g +w_g\right] _E\cap (\Pi ^*)^{-1}(B(w_g,\varepsilon )).
\end{align} 
It is then clear that one can extract a countable subbasis of this family, completing the argument. Now we wish to show the bicontinuity of the stated isomorphism, and to that end we let $(g_\lambda )$ be a net of the Ellis semigroup with associated net $(w_\lambda,\mathfrak{t}_\lambda)$ in the direct product $\mathbb{R}^n\times\mathfrak{T}_{W,\Gamma}$, and $g$ be some Ellis transformation with associated couple $(w,\mathfrak{t})$. Let us first state a useful lemma:

\begin{lem}\label{lem:plain.cone} There exists an $\varepsilon _0>0$ such that, for any $\mathfrak{t}\in \mathfrak{T}_{W,\Gamma }$ and $w\in \mathbb{R}^n_{\mathfrak{t}}= V_\mathfrak{t}+ \Gamma ^*$, we have $$C_{\mathfrak{t}}(w,\varepsilon )\cap \Gamma ^* = \mathsf{C}_{\mathfrak{t}}(w,\varepsilon )\cap \Gamma ^* \; \; \;  \forall \; 0<\varepsilon \leqslant\varepsilon _0.$$
\end{lem}

\begin{proof} Clearly the cone $C_{\mathfrak{t}}(w,\varepsilon )$ contains $ \mathsf{C}_{\mathfrak{t}}(w,\varepsilon )$ for all $\varepsilon >0$. Conversely let $\mathfrak{t}\in \mathfrak{T}_{W,\Gamma }$ be chosen, with associated cone $C_{\mathfrak{t}}$ in $\mathbb{R}^n$ and the direct sum decomposition
$\langle C_{\mathfrak{t}}\rangle= V_{\mathfrak{t}}\oplus D_{\mathfrak{t}}$ provided by Theorem \ref{theo:Senechal}. As $ D_{\mathfrak{t}}\cap \Gamma ^*$ is uniformly discrete in $D_{\mathfrak{t}}$, with $\varepsilon _{\mathfrak{t}}>0$ being some radius of discreteness, we must have $$ \langle C_{\mathfrak{t}}\rangle \cap B(w, \varepsilon _{\mathfrak{t}}) \cap \Gamma ^*= (V_{\mathfrak{t}}+w) \cap B(w, \varepsilon _{\mathfrak{t}})\cap \Gamma ^*$$ for any $w\in V_\mathfrak{t}+ \Gamma ^*$. Hence by intersecting with $C_{\mathfrak{t}}+w$ we obtain 
\begin{align*} C_{\mathfrak{t}}(w,\varepsilon _{\mathfrak{t}})\cap \Gamma ^*= (C_{\mathfrak{t}}+w)\cap (V_{\mathfrak{t}}+w)\cap B(w, \varepsilon _{\mathfrak{t}})\cap \Gamma ^*= \mathsf{C}_{\mathfrak{t}}(w,\varepsilon _{\mathfrak{t}})\cap \Gamma ^* .
\end{align*}
Finally, taking $\varepsilon _0$ to be the minimum over $\varepsilon _{\mathfrak{t}}$, $\mathfrak{t}\in \mathfrak{T}_{W,\Gamma }$, gives the statement.
\end{proof}

Then $ g_\lambda $ converges to $g$ if and only if for any $\varepsilon >0$, which can be chosen less than the constant $\varepsilon _0$ of Lemma \ref{lem:plain.cone}, there is some net of positive real numbers $(\delta _\lambda )$, which can be chosen less than the constant $\varepsilon _0$ as well, such that one has for great enough $\lambda$:
\begin{align*}   \left[ C_{\mathfrak{t}_\lambda} +w_\lambda\right] _E\cap (\Pi ^*)^{-1}(B(w_\lambda ,\delta _\lambda ))\; \subset \; \left[ C_\mathfrak{t} +w\right] _E\cap (\Pi ^*)^{-1}(B(w ,\varepsilon )) . 
\end{align*}
By Lemma \ref{lem:plain.cone}, intersecting with $\Gamma ^*$ leads for great enough $\lambda$ to
\begin{align*} \mathsf{C}_{\mathfrak{t}_\lambda}(w_\lambda ,\delta _\lambda )\cap \Gamma ^* \; \subset \;  \mathsf{C}_{\mathfrak{t}}(w,\varepsilon )\cap \Gamma ^* .
\end{align*}
Now the affine space generated by $\mathsf{C}_{\mathfrak{t}_\lambda}(w_\lambda ,\delta _\lambda )$ is precisely $V_{\mathfrak{t}_\lambda}+w_\lambda $ which contains, since $w_\lambda$ is an allowed translation for $\mathfrak{t}_\lambda$, a dense subset of elements of $\Gamma ^*$. The same occurs about $w$ with respect to $\mathfrak{t}$, and thus we get for great enough $\lambda$ the inclusions
\begin{align*}V_{\mathfrak{t}_\lambda}+w_\lambda \; \subset \; V_{\mathfrak{t}}+w .
\end{align*}

As $\mathsf{C}_{\mathfrak{t}}(w,\varepsilon )$ is a topologically regular open subset of $V_{\mathfrak{t}}+w$, its intersection with $V_{\mathfrak{t}_\lambda}+w_\lambda$ forms an open topologically regular subset of this latter affine space, containing $\mathsf{C}_{\mathfrak{t}_\lambda}(w_\lambda ,\delta _\lambda )\cap \Gamma ^*$. As $\mathsf{C}_{\mathfrak{t}_\lambda}(w_\lambda ,\delta _\lambda )$ is a topologically regular open subset of $V_{\mathfrak{t}_\lambda}+w_\lambda$ as well, taking closure an next interior in $V_{\mathfrak{t}_\lambda}+w_\lambda$ provides for great enough $\lambda$
\begin{align*} \mathsf{C}_{\mathfrak{t}_\lambda}(w_\lambda ,\delta _\lambda )\; \subset \;  \mathsf{C}_{\mathfrak{t}}(w,\varepsilon ) ,
\end{align*}
thus giving the $\Rightarrow $ part of the statement.
\vspace{0.2cm}

Conversely, let us suppose that for any $\varepsilon >0$, which can be chosen less than the constant $\varepsilon _0$ of Lemma \ref{lem:plain.cone}, there is some net of positive real numbers $(\delta _\lambda )$, which can be chosen less than the constant $\varepsilon _0$ as well, such that one has $\mathsf{C}_{\mathfrak{t}_\lambda}(w_\lambda ,\delta _\lambda ) \subset \mathsf{C}_{\mathfrak{t}}(w,\varepsilon ) \subset C_\mathfrak{t}+w$ for great enough $\lambda$. Now the first point is that the net $(w_\lambda)$ converges to $w$ in $\mathbb{R}^n$, and so $g_\lambda$ falls into the inverse image of any ball $B(w, \varepsilon )$ for great enough $\lambda$. Secondly, any $g_\lambda $ has a neighborhood of the form $\left[ C_{\mathfrak{t}_\lambda} +w_\lambda \right] _E\cap (\Pi ^*)^{-1}(B(w_\lambda ,\delta _\lambda ))$, which is contained in the subset $\left[ C_{\mathfrak{t}_\lambda}(w_\lambda ,\delta _\lambda ) \right] _E =\left[ \mathsf{C}_{\mathfrak{t}_\lambda}(w_\lambda ,\delta _\lambda ) \right] _E$ and thus in $\left[ C_\mathfrak{t} +w \right] _E$ for great enough $\lambda$. Combining these two arguments we deduce from the neighborhood basis formula (\ref{ellis.local.basis}) that $g_\lambda$ converges to $g$ in the internal system Ellis semigroup.
\vspace{0.2cm}

This completes the proof of Theorem \ref{theo:principal.interne}.\begin{flushright}
$\square $
\end{flushright}

\section{\large{\textbf{Results on the hull Ellis semigroup and additional algebraic features}}}\label{main}

We arrive at our main result, namely, the algebraic and topological description of the Ellis semigroup for a hull $\mathbb{X}$ of almost canonical model sets together with its $\mathbb{R}^d$--action.

\subsection{The main result}

From Theorem \ref{theo:suspension.ellis}, any transformation $\mathsf{g}$ in the semigroup $E(\mathbb{X},\mathbb{R}^d)$ may be written as $\tilde{g}-s$ where $g$ is a transformation in $E(\Xi ^\Gamma , \Gamma )$ and $s$ a vector of $\mathbb{R}^d$, and with $g$ uniquely defined up to an element of $\Gamma$. Thus we may associate to any transformation $\mathsf{g}=\tilde{g}-s$ the cone type of any underlying transformation $g\in E(\Xi ^\Gamma , \Gamma )$, which we write $\mathfrak{t}_\mathsf{g}$, thus providing a semigroup morphism from the hull Ellis semigroup $E(\mathbb{X},\mathbb{R}^d)$ to the nontrivial face semigroup $\mathfrak{T}_{W,\Gamma }$. We are now able to formulate the main result of this work, which is completely deduced from Theorems \ref{theo:suspension.ellis} and \ref{theo:principal.interne}:

\begin{theo}\label{theo:principal} The mapping that associates to any transformation $\mathsf{g}$ the couple $(z_\mathsf{g},\mathfrak{t}_\mathsf{g})$ establishes an isomorphism between the Ellis semigroup $E(\mathbb{X},\mathbb{R}^d)$ and the subsemigroup of the direct product $\left[ \mathbb{R}^{n+d}\right] _{\Sigma} \times \mathfrak{T}_{W,\Gamma }$ given by
\begin{align*}\bigsqcup _{\mathfrak{t}\in \mathfrak{T}_{W,\Gamma }} \left[ \langle \mathsf{C}_\mathfrak{t}\rangle \times\mathbb{R}^d\right]_{\Sigma} \times \left\lbrace \mathfrak{t}\right\rbrace .
\end{align*}

Moreover, this isomorphism becomes a homeomorphism when the above union is equipped with the following convergence class: $(z_\lambda,\mathfrak{t}_\lambda)\longrightarrow (z,\mathfrak{t})$  if and only if one can write $z_\lambda =[w_\lambda,s_\lambda]_{\Sigma}$ and $z =[w,s]_{\Sigma}$ such that

\begin{itemize}
\item[$\mathrm{(1)}$] $s_\lambda \longrightarrow s$ in $\mathbb{R}^d$,
\vspace{0.1cm}

\item[$\mathrm{(2)}$] $\forall \varepsilon >0, \exists \, \delta _\lambda >0 \text{ such that  } \; \mathsf{C}_{\mathfrak{t}_\lambda}(w_\lambda,\delta _\lambda)\subset \mathsf{C}_{\mathfrak{t}}(w,\varepsilon ) \; \text{  for large enough }\lambda $ in $\mathbb{R}^n$.
\vspace{0.2cm}
\end{itemize}

Finally, the Ellis semigroup $E(\mathbb{X},\mathbb{R}^d)$ has a first countable topology, and the dynamical system $(\mathbb{X},\mathbb{R}^d)$ is tame.
\end{theo}

\subsection{Additional algebraic features}

\paragraph{7.2.1 $\;$ Invertible Ellis transformations} One can naturally ask whether there are transformations in the hull Ellis semigroup which are invertible but not homeomorphisms given by the $\mathbb{R}^d$--action. It turns out that the answer is no: We have seen that any cone type $\mathfrak{t}\in \mathfrak{T}_{W,\Gamma}$ is idempotent, and thus an invertible transformation must corresponds to a couple of the form $(z,\mathfrak{o})$ where $\mathfrak{o}$ is the identity cone type in $\mathfrak{t}\in \mathfrak{T}_{W,\Gamma}$. Since the cone with cone type $\mathfrak{o}$ is precisely the trivial cone $\left\lbrace 0\right\rbrace $, its associated plain cone $\mathsf{C}_\mathfrak{o}$ is nothing but $\left\lbrace 0\right\rbrace $ and Theorem \ref{theo:principal} ensures that $z$ must be an element of the form $[0,s]_{\Sigma}$ in $\left[  \left\lbrace 0\right\rbrace \times \mathbb{R}^d\right] _{\Sigma}$. It follows that the underlying transformation is the homeomorphism arising from translation by the vector $s\in \mathbb{R}^d$.

\paragraph{7.2.2 $\;$ Range of Ellis transformations} It is natural to define on the Ellis semigroup $E(\mathbb{X},\mathbb{R}^d)$ a preorder by letting $\mathsf{g}\leqslant \mathsf{g}'$ whenever the \textit{range} of the mapping $\mathsf{g}$ is contained in that of $\mathsf{g}'$. By range we mean here the subset $r(\mathsf{g}):= \mathbb{X}.\mathsf{g}$ of the hull $\mathbb{X}$. When one considers idempotent transformations $\mathsf{q}$ and $\mathsf{q}'$ then it is easy to show that $\mathsf{q}\leqslant \mathsf{q}'$ when and only when one has $\mathsf{q} = \mathsf{q}.\mathsf{q}'$, thus turning this preorder into algebraic terms in this particular setting. In the case of an almost canonical hull Ellis semigroup we are able to describe this preorder in a quite elegant manner:

\begin{prop}\label{prop:range} For any transformations of $E(\mathbb{X},\mathbb{R}^d)$ we have the equivalence
\begin{align*} \mathsf{g}\leqslant \mathsf{g}' \quad \Longleftrightarrow \quad C_\mathsf{g} \leqslant C_{\mathsf{g}'}.
\end{align*}
\end{prop}

The proposition above asserts that the range of $\mathsf{g}$ is contained into the range of $\mathsf{g}'$ if and only if the cone $C_{\mathsf{g}'}$ is equal or a lower dimensional facet of the cone $C_\mathsf{g}$.\\

\begin{proof} Let $\mathsf{g}$ and $ \mathsf{g}'$ be chosen. Each are element of a subgroup respectively given by $\left[ \langle \mathsf{C}_{\mathfrak{t}_\mathsf{g}}\rangle \times\mathbb{R}^d\right]_{\Sigma} \times \left\lbrace \mathfrak{t}_\mathsf{g}\right\rbrace $ and $\left[ \langle \mathsf{C}_{\mathfrak{t}_{\mathsf{g}'}}\rangle \times\mathbb{R}^d\right]_{\Sigma} \times \left\lbrace \mathfrak{t}_{\mathsf{g}'}\right\rbrace $, and thus one can see that $r(\mathsf{g})= r([0]_{\Sigma}\times\left\lbrace \mathfrak{t}_\mathsf{g}\right\rbrace )$ and that $r(\mathsf{g}')= r([0]_{\Sigma}\times\left\lbrace \mathfrak{t}_{\mathsf{g}'}\right\rbrace )$. From what have just been said it becomes clear that $\mathsf{g}\leqslant \mathsf{g}'$ if and only if $\mathfrak{t}_{\mathsf{g}}=\mathfrak{t}_{\mathsf{g}}.\mathfrak{t}_{\mathsf{g}'}$, which exactly means that the cone $C_{\mathsf{g}'}$ is equal or a lower-dimensional facet of the cone $C_\mathsf{g}$, or equivalently  $C_\mathsf{g} \leqslant C_{\mathsf{g}'}$.
\end{proof}

\paragraph{7.2.3 $\;$ Ideals} The general theory of Ellis semigroups gives great importance to the ideal theory of an Ellis semigroup. In the case of a almost canonical hull Ellis semigroup it is easy to prove the proposition stated below, showing that the ideal theory of the hull Ellis semigroup reduces to the ideal theory of the semigroup $\mathfrak{T}_{W,\Gamma}$:

\begin{prop}\label{prop:min} Each right ideal $\mathfrak{M}$ of the nontrivial face semigroup $\mathfrak{T}_{W,\Gamma}$ defines a right ideal of the Ellis semigroup $E(\mathbb{X},\mathbb{R}^d)$ by
\begin{align*}\bigsqcup _{\mathfrak{t}\in \mathfrak{M}} \left[ \langle \mathsf{C}_\mathfrak{t}\rangle \times\mathbb{R}^d\right]_{\Sigma} \times \left\lbrace \mathfrak{t}\right\rbrace 
\end{align*}
and conversely each right ideal of $E(\mathbb{X},\mathbb{R}^d)$ arises in this manner.
\end{prop}

We can in particular easily identify the unique minimal ideal of $E(\mathbb{X},\mathbb{R}^d)$: This latter is isomorphic with the direct product $\left[ \mathbb{R}^{n+d}\right] _{\Sigma} \times \mathfrak{M}^{\, ch}$ where $\mathfrak{M}^{\, ch}$ is the family of cone types associated with the \textit{chambers} of the stratification defined by the collection of hyperplanes used to construct the face semigroup $\mathfrak{T}_W$.

\subsection{An explicit computation}

We consider the hull $\mathbb{X}_{oct}$ associated to the real cut and project scheme and octagonal window presented in \ref{subsection:explicit.example}. The associated family of linear hyperplanes parallel to faces of the window (or its reversed set) is described, in the orthonormal basis $(e_1^*, e_2^*)$ of the internal space $\mathbb{R}^2_{int}$, as
\begin{align*} H_1:= \langle v_1\rangle = \langle v_2-v_4\rangle , \quad H_2:= \langle v_2\rangle= \langle v_1+v_3\rangle  \\ H_3:= \langle v_3\rangle =\langle v_2+v_4\rangle ,  \quad H_4:= \langle v_4\rangle =\langle v_1-v_3\rangle ,
\end{align*}

where
\begin{align*} v_1:=e_1^*, \quad v_2:=(e_1^*+e_2^*)\diagup \sqrt{2}, \quad v_3:=e_2^*, \quad v_4:=(e_2^*-e_1^*)\diagup\sqrt{2};
\end{align*}

see Figure 4.

\includegraphics[trim = 7cm 18.5cm 8.5cm 4.8cm]{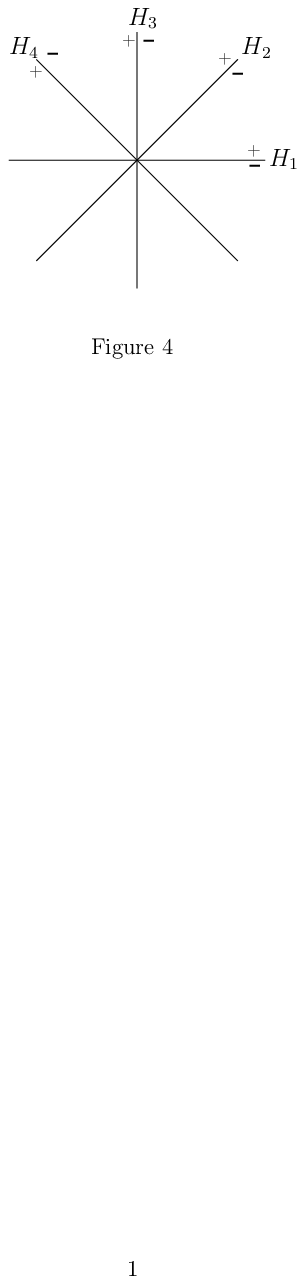}
\includegraphics[trim = 6cm 18.6cm 8.5cm 6cm]{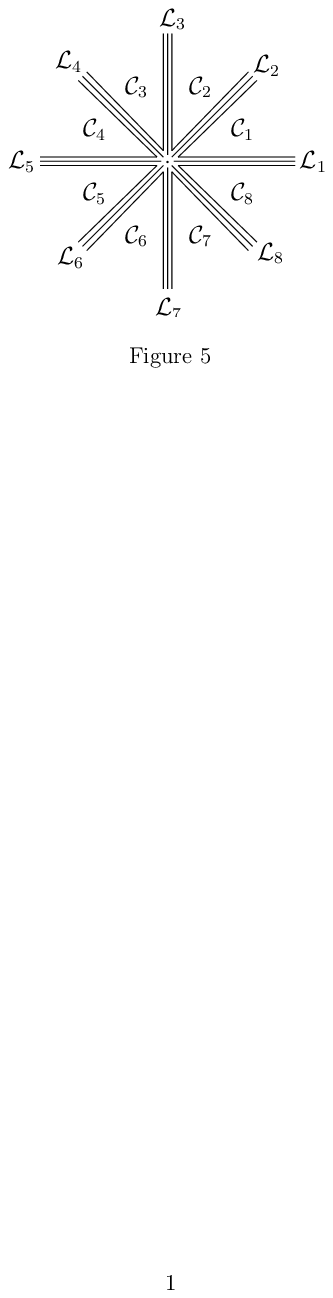}

The stratification obtained from these hyperplanes is of the form in Figure 5.
The internal space $\mathbb{R}^2_{int}$ is partitioned into $17$ different cones: the singleton $\lbrace 0\rbrace $, eight half-lines $\lbrace L_1,...,L_8\rbrace $ pointed at $0$ though not containing it which we label $L_i$, $L_{i+4}\, \subset H_i$ for $1\leqslant i\leqslant 4$, and eight chambers $\lbrace C_1,...,C_8\rbrace $, each consisting of an $\frac{1}{8}^{th}$ part of the space and being open cones pointed at $0$.
\vspace{0.2cm}

Now the stabilizers $Stab_{\Gamma}(H_i)$ are dense in $H_i$ for each index $1\leqslant i\leqslant 4$, and we deduce that each cone of this stratification is nontrivial, and moreover equal to its associated plain cone. Thus $\langle \mathsf{C}_\mathfrak{o}\rangle =\lbrace 0\rbrace$ as usual, whereas $\langle \mathsf{C}_{\mathfrak{t}_{L_i}}\rangle = \langle \mathsf{C}_{\mathfrak{t}_{L_{i+4}}}\rangle = H_i$ for each value $1\leqslant i\leqslant 4$, and $\langle \mathsf{C}_{\mathfrak{t}_{C_i}}\rangle = \mathbb{R}^2$ for each index $1\leqslant i \leqslant 8$.\\

Consequently, the hull Ellis semigroup $E(\mathbb{X}_{oct}, \mathbb{R}^2) $ is in this case obtained as 
\begin{align*}\left( \bigsqcup _{i=1}^8\left[ \mathbb{R}^4\right] _{\mathbb{Z}^4}\times \lbrace \mathfrak{t}_{C_i}\rbrace \right) \bigsqcup \left( \bigsqcup _{i=1}^4 \left[ H_i \times \mathbb{R}^2\right]  _{\mathbb{Z}^4}\times \lbrace \mathfrak{t}_{L_i}, \mathfrak{t}_{L_{i+4}}\rbrace \right)  \bigsqcup \mathbb{R}^2 .
\end{align*}

\section{\large{\textbf{The Ellis action on the hull}}}\label{action}

\subsection{A further look on cones}

We saw in Section $5$ that to any model set $\Lambda$ of the internal system can be associated a cone $C_\Lambda$, that is, an open connected cone pointed at $0$ with boundary delimited by hyperplanes of a subfamily $\mathfrak{H}_{w_\Lambda}$ of $\mathfrak{H}_W$. Moreover each such cone admits a unique cone type $\mathfrak{c}_\Lambda$ with domain $\mathfrak{H}_{w_\Lambda}$, and there can be only finitely many such cone types, whose family is denoted $\mathfrak{C}$. Now if one look at some model set $\Lambda _0$ in the hull $\mathbb{X}$ then it always can be written $\Lambda _0=\Lambda -t$, where $\Lambda$ lies in the internal system and $t$ is a vector of $\mathbb{R}^d$. This presentation is unique up to a translation of both the model set $\Lambda $ and the vector $t$ by some $\gamma \in \Gamma$. Thus one may without misunderstanding define the cut type $\mathfrak{H}_{z_{\Lambda_0}}$ and the cone type $\mathfrak{c}_{\Lambda _0}$ with domain $\mathfrak{H}_{z_{\Lambda_0}}$ to be the ones associated with $\Lambda \in \Xi ^\Gamma$ in the decomposition $\Lambda _0=\Lambda -t$. We may then describe the hull, as was already done by Le \cite{Le}, as follows:

\begin{theo}\label{theo:topologie.par.cones} The mapping associating to any model set $\Lambda $ the couple $(z_\Lambda,\mathfrak{c}_\Lambda)$ establishes a bijective correspondence between the hull $\mathbb{X}$ and
\begin{align*}\left\lbrace (z,\mathfrak{c})\in \left[ \mathbb{R}^{n+d}\right]  _{\Sigma}\times\mathfrak{C} \; \vert \; dom(\mathfrak{c})= \mathfrak{H}_{z} \right\rbrace  .
\end{align*}
\end{theo}

\begin{proof} From what has been just said it is sufficient to prove that the mapping associating to any model set $\Lambda $ in $\Xi ^\Gamma$ the couple $(w_\Lambda,\mathfrak{c}_\Lambda)$ establishes a bijective correspondence between the internal system $\Xi ^\Gamma$ and
\begin{align*}\left\lbrace (w,\mathfrak{c})\in  \mathbb{R}^n \times\mathfrak{C} \; \vert \; dom(\mathfrak{c})= \mathfrak{H}_{w} \right\rbrace  .
\end{align*}
First from the very construction of the cone type $\mathfrak{c}_\Lambda$ associated with any $\Lambda \in \Xi ^\Gamma$ this association is well defined. By the arguments used in the proof of Proposition \ref{prop:topology.internal.system}, each model set $\Lambda\in \Xi ^\Gamma$ is the limit of a filter base (\ref{nei}) which only depends on the couple $(w_\Lambda, \mathfrak{c}_\Lambda)$, and thus the association is one-to-one.
Moreover this association is onto: If $(w, \mathfrak{c})$ is a couple with $dom(\mathfrak{c})$ equal to $\mathfrak{H}_w$, then consider the family of subsets $ [C_{\mathfrak{c}}+w]_\Xi \cap \Pi ^{-1}(B(w,\varepsilon))$ of the internal system. Each such set contains some nonsingular model sets, and thus forms a filter base in $\Xi ^\Gamma$. As $\Pi$ is a proper map this filter base is eventually contained in a compact subset of the form $\Pi ^{-1}(\overline{B}(w,\varepsilon))$ and thus admits an accumulation element $\Lambda$. This latter must satisfies $\Pi (\Lambda)=w_\Lambda =w$ and $C_\Lambda= C_{\mathfrak{c}}$ on the other hand. But as the domains of $\mathfrak{c}$ and $\mathfrak{c}_\Lambda$ are both equal to the cut type of $w_\Lambda =w$ the couple $(w_\Lambda,\mathfrak{c}_\Lambda)$ is nothing but $(w,\mathfrak{c})$, showing that the association is onto.
\end{proof}

\subsection{The Ellis action}

We wish to use here the descriptions of the hull obtained in the above paragraph and that of its Ellis semigroup performed in Theorem \ref{theo:principal}. To this end we set an action of the nontrivial face semigroup $\mathfrak{T}_{W,\Gamma}$ on the family $\mathfrak{C}$ of cone types introduced above:

\vspace{0.2cm}
For $\mathfrak{c}\in \mathfrak{C}$ and $\mathfrak{t}\in \mathfrak{T}_{W,\Gamma}$ let us define a map $\mathfrak{H}_W \longrightarrow \left\lbrace -,+,\infty \right\rbrace $ as
\begin{align*}
\mathfrak{c}.\mathfrak{t}(H):= \begin{cases} \mathfrak{c}(H) \; \; $ if $\mathfrak{t}(H)= 0 ,\\ \mathfrak{t}(H)\; \, \; $ else. $  \end{cases}
\end{align*}
This definition is not properly an action of $\mathfrak{T}_{W,\Gamma}$ on $\mathfrak{C}$ as the resulting map may not be a cone issuing from any model set of the hull $\mathbb{X}$. However it allows us to recover the Ellis action as follows:

\begin{prop}\label{prop:ellis.action} The Ellis action $\mathbb{X}\times E(\mathbb{X}, \mathbb{R}^d )\longrightarrow \mathbb{X}$ obtains as
\begin{align*}(z,\mathfrak{c}).(z',\mathfrak{t})= (z+z',\mathfrak{c}') \quad \text{where} \quad \mathfrak{c}'(H):= \begin{cases} \mathfrak{c}.\mathfrak{t}(H) \quad \text{ if }H\in \mathfrak{H}_{z+z'}, \\ \infty  \qquad \; \; \text{ else. } \end{cases}
\end{align*}
\end{prop}

\subsection{An illustration of the Ellis action}

In order to illustrate the Ellis action described as above, we focus here on the example of the hull $\mathbb{X}_{oct}$ associated with the data given in Section \ref{subsection:explicit.example}. More precisely we won't describe the action of any transformation but rather the one of the idempotent transformations (as the other part is only a shifting in the parametrization torus $\left[  \mathbb{R}^4\right] _{\mathbb{Z}^4}$). Moreover it can be checked that the idempotent Ellis transformations are precisely those Ellis transformations mapped onto $0\in \left[  \mathbb{R}^4\right] _{\mathbb{Z}^4}$ under $\pi ^*$, or equivalently, those which preserve fibers in $\mathbb{X}_{oct}$ with respect to the parametrization map $\pi$. Here we won't describe the Ellis action of these idempotents at any model set, but rather on the single fiber above $0\in \left[  \mathbb{R}^4\right] _{\mathbb{Z}^4}$, any other fiber can be treated in the same manner.
\vspace{0.2cm}

First we need to know the cut type of $0$: it is easily checked that $\mathfrak{H}_0= \mathfrak{H}_{W_{oct}}= \left\lbrace H_1, H_2, H_3, H_4\right\rbrace $, so that the fiber above $0$ in the hull consists of eight model sets $ \; \;  $ $\left\lbrace \Lambda _{C_1},..,\Lambda _{C_8}\right\rbrace $, each associated with some cone which is in this particular case a \textit{chamber} among $\left\lbrace C_1,..,C_8\right\rbrace $. Then we can compute the action of any of the $17$ idempotent transformations $[0]_{\mathbb{Z}^4}\times\left\lbrace \mathfrak{t} \right\rbrace $ , $\mathfrak{t} \in \mathfrak{T}_{W_{oct}}$:\\

The identity map, given by $[0]_{\mathbb{Z}^4}\times\left\lbrace \mathfrak{o} \right\rbrace $, preserves any of the eight model sets, whereas any idempotent map $[0]_{\mathbb{Z}^4}\times\left\lbrace \mathfrak{t}_{C_i} \right\rbrace $ associated with the chamber $C_i$ maps all of these model sets onto a single one, namely $\Lambda _{C_i}$. For an idempotent map of the form $[0]_{\mathbb{Z}^4}\times\left\lbrace \mathfrak{t}_{L_i} \right\rbrace $ with $L_i$ some half-line contained in the hyperplane $H_i$, each model set with associated cone belonging to the side $\pm$ of $H_i$ is mapped onto the unique model set whose cone belongs to the same side $\pm$ of $H_i$ and has $L_i$ in its boundary. Therefore these transformations have two distinct model sets of this fiber in their range, namely these which have $L_i$ in the boundary of their associated cone.
\vspace{0.2cm}

\textbf{Acknowledgement} I wish to express my deep gratitude to my advisor Johannes Kellendonk for all its valuable suggestions, discussions and comments about the present article.




\end{document}